\documentclass[a4paper,11pt]{amsart}
\usepackage{latexsym}
\usepackage{amscd}
\usepackage{amssymb}
\usepackage{amsfonts}
\usepackage{mathrsfs}
\usepackage[all]{xy}
\usepackage{amsmath, pb-diagram}
\usepackage{enumerate}

\DeclareMathAlphabet{\eusm}{OT1}{eusm}{m}{n}
\newtheorem{theor}{Theorem}[section]
\newtheorem{prop}[theor]{Proposition}
\newtheorem{defi}[theor]{Definition}
\newtheorem{cor}[theor]{Corollary}
\newtheorem{lem}[theor]{Lemma}
\newtheorem{exam}[theor]{Example}
\newtheorem{remark}[theor]{Remark}

\def\Ker{\mbox{Ker\/}}

\newcommand{\Hom}{{\rm Hom}}

\newcommand{\A}{\mathcal A}

\begin{document}
\title[Split objects in abelian categories]{Split objects with respect to a fully invariant \\ short exact sequence in abelian categories}

\author[S. Crivei]{Septimiu Crivei}
\address {Faculty of Mathematics and Computer Science, Babe\c{s}-Bolyai University, Str. M. Kog\u alniceanu 1,
400084 Cluj-Napoca, Romania}
\email{crivei@math.ubbcluj.ro}

\author[D. Kesk\.{i}n T\"ut\"unc\"u]{Derya Kesk\.{i}n T\"ut\"unc\"u}
\address {Department of Mathematics, Hacettepe University, 06800 Beytepe, Ankara, Turkey} \email{keskin@hacettepe.edu.tr}

\author[R. Tribak]{Rachid Tribak}
\address {Centre R\'{e}gional des M\'{e}tiers de L'Education et de la Formation (CRMEF)-Tanger \\
Avenue My Abdelaziz, Souani, B.P.:3117 \\ Tangier 90000, Morocco} \email{tribak12@yahoo.com}


\begin{abstract} We introduce and investigate (dual) relative split objects
with respect to a fully invariant short exact sequence in abelian categories.
We compare them with (dual) relative Rickart objects,
and we study their behaviour with respect to direct sums and 
classes all of whose objects are (dual) relative split.
We also introduce and study (dual) strongly relative split objects. 
Applications are given to Grothendieck categories, module and comodule categories.
\end{abstract}

\dedicatory{Dedicated to the memory of Professor John Clark}

\date{March 13, 2018}

\maketitle


\section{Introduction}

Right Rickart rings have the root in the work of Rickart \cite{Rickart} on certain Banach algebras. 
In an arbitrary ring-theoretic setting they were considered by Maeda \cite{Maeda}, who defined them as 
rings in which the right annihilator of any element is generated by an idempotent. 
Independently, Hattori \cite{Hattori} defined right p.p. rings as rings in which every principal right ideal is projective.
Later on, it turned out that the two concepts are actually equivalent, and their theory was further 
developed by many mathematicians. Some classical examples of right Rickart rings include 
right semihereditary rings, Baer rings and von Neumann regular rings. 

A natural development of the theory of Rickart rings was towards module theory. 
Thus, Lee, Rizvi and Roman introduced and extensively studied Rickart modules and dual Rickart modules in a series of papers
\cite{LRR10,LRR11,LRR12}, which extended the theory of Baer modules \cite{RR04} and dual Baer modules \cite{KT}. 
A further generalization was considered by Crivei and K\"or \cite{C-K}, 
who investigated relative Rickart objects and dual relative Rickart objects in abelian categories (also, see \cite{CO,CO1}). 
If $M$ and $N$ are objects of an abelian category, then $N$ is called $M$-Rickart if for every morphism $f:M\to N$, 
${\rm ker}(f)$ is a section, while $N$ is called dual $M$-Rickart if for every morphism $f:M\to N$, 
${\rm coker}(f)$ is a retraction. Their motivation was to set a unified theory of relative Rickart objects 
with versatile applications, which allows one to deduce naturally properties of dual relative Rickart objects 
(by the duality principle), relative regular objects (which are relative Rickart and dual relative Rickart)
in the sense of D\u asc\u alescu, N\u ast\u asescu, Tudorache and D\u au\c s \cite{DNT,DNTD} as well as 
relative Baer objects (as particular relative Rickart objects) and dual relative Baer objects (by the duality principle) 
\cite{C-K,CO,CO1}. In recent years, Rickart modules and dual Rickart modules were generalized 
to (dual) $t$-Rickart modules by Asgari and Haghany \cite{AH},
(dual) $T$-Rickart modules by Ebrahimi Atani, Khoramdel and Dolati Pish Hesari \cite{Atani12, Atani16}
and (dual) $F$-inverse split modules by Ungor, Halicioglu and Harmanci \cite{Ungor,Ungor2}, 
the latter developing a theory using arbitrary fully invariant submodules. 

In the present paper we study (dual) relative split objects
with respect to a fully invariant short exact sequence in abelian categories. 
Note that any preradical $r$ of a category gives rise to a fully invariant short exact sequence 
$0\to r(M)\to M\to M/r(M)\to 0$ for every object $M$.  
If $M$ is an object and $0\to F\stackrel{i}\to N\stackrel{d}\to C\to 0$ is a fully invariant short exact sequence, 
then $N$ is $M$-$F$-split if and only if for every morphism $g:M\to N$, ${\rm ker}(dg)$ is a section, 
while $N$ is dual $M$-$F$-split if and only if for every morphism $g:N\to M$, ${\rm coker}(gi)$ is a retraction. 
For $M=N$, one has the notion of (dual) self-$F$-split object. 
We also introduce a strong version of (dual) relative split object, which generalizes (dual) strong Rickartness
in the sense of \cite{CO1,Atani14,WangY}.  
Relative split objects generalize all the above Rickart-type modules as well as extending modules. 
Note that $N$ is $M$-Rickart if and only if $N$ is $M$-$0$-split, while
$N$ is dual $M$-Rickart if and only if $N$ is dual $M$-$M$-split. 
Moreover, $M$ is (dual) self-$F$-split if and only if $M\cong F\oplus C$ and $C$ is self-Rickart ($F$ is dual self-Rickart). 
When $F=r(M)$ for some (pre)radical $r$ of $\A$, this shows that self-$F$-split and dual self-$F$-split objects 
are some particular objects for which their (pre)torsion part splits off (see the splitting problem discussed by 
Chase \cite{Chase}, Goodearl \cite{Goodearl}, Kaplansky \cite{Kap}, N\u ast\u asescu and Torrecillas \cite{NT} 
or Teply \cite{Teply}). Applications are given to Grothendieck categories, module and comodule categories.

Our results will usually have two parts, out of which we only prove the first one, 
the second one following by the duality principle in abelian categories. 
In Section 2 we present some needed terminology and properties 
on fully invariant short exact sequences in abelian categories. In Section 3 we introduce 
(strongly) relative split objects with respect to a fully invariant short exact sequence 
$0\to F\to N\to C\to 0$ in an abelian category $\A$. We show that 
an object $M$ of $\A$ is strongly self-$F$-split if and only if  
$M$ is self-$F$-split and every direct summand of $M$ which contains $F$ is fully invariant.
We prove that every (strongly) self-$F$-split object has the summand intersection property 
for (fully invariant) direct summands containing $F$. One of our main results shows that 
if $e:M\to M'$ is an epimorphism, $m:N'\to N$ is a monomorphism,   
the inclusion monomorphism $u:F\cap N'\to N'$ is fully invariant and
$N$ is (strongly) $M$-$F$-split, then $N'$ is (strongly) $M'$-$(F\cap N')$-split. 
In particular, $N$ is (strongly) $M$-$F$-split if and only if for every direct summand $M_1$ of $M$ and 
for every direct summand $N_1$ of $N$, $N_1$ is (strongly) $M_1$-$(F\cap N_1)$-split.
In Section 4 we compare relative self-$F$-split objects and relative self-Rickart objects. 
We prove that $M$ is (strongly) self-$F$-split if and only if $M\cong F\oplus C$ and $C$ is (strongly) self-Rickart.
We also show that $M$ is strongly self-$F$-split if and only if $M$ is self-$F$-split and the ring 
${\rm End}_{\mathcal{A}}(C)$ is abelian. In Section 5 we deal with coproducts of relative split objects. 
If $M$ and $N$ are objects of $\A$, $N=\bigoplus_{k=1}^n N_k$ is a direct sum decomposition, 
and $0\to F\to N\to C\to 0$ is a fully invariant short exact sequence in $\A$, then it is shown that 
$N$ is (strongly) $M$-$F$-split if and only if $N_k$ is (strongly) $M$-$(F\cap N_k)$-split for every $k\in \{1,\dots,n\}$.
Also, if $r$ is a preradical of $\A$, $M$ is an object of $\A$ having the strong summand intersection property 
for (fully invariant) direct summands containing $r(M)$ and $(N_k)_{k\in K}$ is a family of objects of $\A$,
then $\bigoplus_{k\in K} N_k$ is (strongly) $M$-$r(\bigoplus_{k\in K} N_k)$-split if and only if 
$N_k$ is (strongly) $M$-$r(N_k)$-split for every $k\in K$. 
In Section 6 we study classes all of whose objects are self-$F$-split. We characterize spectral categories, 
locally finitely generated Grothendieck categories which are $V$-categories or regular categories, 
(semi)hereditary categories with enough projectives and co(semi)hereditary categories with enough injectives
in terms of self-$F$-splitness or dual self-$F$-splitness of objects of certain classes. 
If an abelian category $\A$ has a generator $G$ and enough injectives, and $r$ is a radical of $\A$, 
then we characterize when $G/r(G)$ is (semi)hereditary. 
Also, if $\A$ has an injective cogenerator $G$ and $r$ is a preradical of $\A$, then we establish equivalent conditions 
for $G$ to be (strongly) self-$r(G)$-split. Finally, in Section 7 we give some further applications 
to module and comodule categories. Among them, we show that every (strongly) extending module $N$ 
is (strongly) self-$F$-split, where $F$ is the second singular submodule of $N$.

Let us finally note that the transfer via functors between abelian categories 
of the (dual) relative splitness of objects with respect to a fully invariant short exact sequence
is studied in the separate paper \cite{CKT2}.

\section{Fully invariant short exact sequences}

Let $\mathcal{A}$ be an abelian category. For every morphism $f:M\to N$ in $\mathcal{A}$ we denote 
its kernel, cokernel, coimage and image by ${\rm ker}(f):{\rm Ker}(f)\to M$, 
${\rm coker}(f):N\to {\rm Coker}(f)$, ${\rm coim}(f):M\to {\rm Coim}(f)$ 
and ${\rm im}(f):{\rm Im}(f)\to N$ respectively. Since $\mathcal{A}$ is an abelian category, 
one has ${\rm Coim}(f)\cong {\rm Im}(f)$. 
For a short exact sequence $0\to A\to B\to C\to 0$ in $\A$, we sometimes write $C=B/A$.
Recall that a morphism $f:M\to N$ is called a \emph{section} (\emph{retraction}) 
if there exists a morphism $f':N\to M$ such that $f'f=1_M$ ($ff'=1_N$).

Now we recall some terminology on fully invariant kernels from \cite{CO1}, 
and we establish a series of properties, which generalize corresponding properties of 
fully invariant submodules of modules, and will be needed in the next sections of the paper.

\begin{defi}{\cite[Definition~2.2]{CO1}} \rm Let $\mathcal{A}$ be an abelian category.
\begin{enumerate}
\item A kernel (monomorphism) $i:K\to M$ in $\mathcal{A}$ is called \emph{fully invariant} if for every morphism $h:M\to M$,
$hi$ factors through $i$.
\item A cokernel (epimorphism) $d:M\to C$ in $\mathcal{A}$ is called \emph{fully coinvariant} if for every morphism $h:M\to M$,
$dh$ factors through $d$.
\end{enumerate}
\end{defi}

\begin{lem}{\cite[Lemma~2.5]{CO1}} \label{l:eq} Let $\mathcal{A}$ be an abelian category.
Then a kernel $i:K\to M$ in $\mathcal{A}$ is fully invariant
if and only if its cokernel $d:M\to C$ is fully coinvariant.
\end{lem}

\begin{defi} \rm Let $\mathcal{A}$ be an abelian category. A short exact sequence
$0\to A\stackrel{i}\to B\stackrel{d}\to C\to 0$ in $\mathcal{A}$ is called \emph{fully invariant} if
$i$ is fully invariant, or equivalently, $d$ is fully coinvariant.
\end{defi}

\begin{exam}{\cite[Example~2.3]{CO1}} \rm Consider the category ${\rm Mod}(R)$ 
of unitary right modules over a ring $R$ with identity. Then a kernel $i:K\to M$ is
fully invariant if and only if for every endomorphism $h:M\to M$,
$hi=i\alpha$ for some homomorphism $\alpha:K\to K$ if and only if
for every endomorphism $h:M\to M$, $h(K)={\rm Im}(hi)\subseteq {\rm Im}(i)=K$ if and only if
$K$ is a fully invariant submodule of $M$.
\end{exam}

\begin{lem}{\cite[Lemma~2.4]{CO1}} \label{l:trans} Let $\mathcal{A}$ be an abelian category.
\begin{enumerate}
\item The composition of two fully invariant kernels in $\A$ is a fully invariant kernel.
\item The composition of two fully coinvariant cokernels in $\A$ is a fully coinvariant cokernel.
\end{enumerate}
\end{lem}

\begin{prop} \label{p:fids} Let $\A$ be an abelian category. Let $0\to F\stackrel{i}\to M\stackrel{d}\to M/F\to 0$
be a fully invariant short exact sequence in $\A$ and $M=M_1\oplus M_2$. Then:
\begin{enumerate}
\item $F\cong (F\cap M_1)\oplus (F\cap M_2)$.
\item $M/F\cong M/(F+M_1)\oplus M/(F+M_2)$.
\end{enumerate}
\end{prop}

\begin{proof} (1) For $l=1,2$, denote by $u_l:M_l\to M$ the canonical injection,
and by $p_l:M\to M_l$ the canonical projection. Also, for $l=1,2$, denote by
$j_l:F\cap M_l\to M_l$ and $k_l:F\cap M_l\to F$ the inclusion monomorphisms.
The universal property of the coproduct yields the monomorphism
$\left[\begin{smallmatrix} k_1&k_2 \end{smallmatrix}\right]:(F\cap M_1)\oplus (F\cap M_2)\to F$.

Since $i:F\to M$ is fully invariant, there exist morphisms $\alpha_1,\alpha_2:F\to F$ such that $u_1p_1i=i\alpha_1$
and $u_2p_2i=i\alpha_2$. Then $i(\alpha_1+\alpha_2)=u_1p_1i+u_2p_2i=i$, hence $\alpha_1+\alpha_2=1_F$, because $i$ is a monomorphism.
For $l=1,2$ the following square is a pullback:
$$\SelectTips{cm}{}
\xymatrix{
F\cap M_l \ar[r]^-{j_l} \ar[d]_{k_l} & M_l \ar[d]^{u_l} \\
F \ar[r]_i & M }$$ For $l=1,2$, by the pullback property, there
exists a morphism $\gamma_l:F\to F\cap M_l$ such that
$j_l\gamma_l=i_l$ and $k_l\gamma_l=\alpha_l$. Then
$\left[\begin{smallmatrix} k_1&k_2 \end{smallmatrix}\right]
\left[\begin{smallmatrix} \gamma_1 \\ \gamma_2
\end{smallmatrix}\right]=k_1\gamma_1+k_2\gamma_2=\alpha_1+\alpha_2=1_F$.
Hence $\left[\begin{smallmatrix} k_1&k_2 \end{smallmatrix}\right]$
is a retraction, and so it is an isomorphism. This shows that
$F\cong (F\cap M_1)\oplus (F\cap M_2)$.
\end{proof}

\begin{prop}\label{p:FG2} Let $\A$ be an abelian category. Let $0\to F\stackrel{i}\to M\stackrel{d}\to M/F\to 0$ 
be a fully invariant short exact sequence and $0\to G\stackrel{j}\to M\stackrel{p}\to M/G\to 0$ 
a short exact sequence in $\A$.
\begin{enumerate}
\item Let $u:F\cap G\to G$ be the inclusion monomorphism. 
\begin{enumerate}[(i)]
\item Assume that the above second short exact sequence is also fully invariant. 
Then the inclusion monomorphism $ju:F\cap G\to M$ is fully invariant.
\item Assume that every morphism $G\to G$ can be extended to a morphism $M\to M$ in $\A$. 
Then $u:F\cap G\to G$ is fully invariant.
\end{enumerate}
\item Let $q:M/G\to M/(F+G)$ be the induced epimorphism.
\begin{enumerate}[(i)]
\item Assume that the above second short exact sequence is also fully invariant.
Then the induced epimorphism $qp:M\to M/(F+G)$ is fully coinvariant.
\item Assume that every morphism $M/G\to M/G$ can be lifted to a morphism $M\to M$ in $\A$. 
Then $q:M/G\to M/(F+G)$ is fully coinvariant.
\end{enumerate}
\end{enumerate}
\end{prop}

\begin{proof} (1) Denote by $k:F\cap G\to F$ the inclusion monomorphism. 

(i) Let $f:M\to M$ be a morphism in $\A$. Since $i$ and $j$ are fully invariant kernels, there exist
morphisms $\alpha:F\to F$ and $\beta:G\to G$ such that $i\alpha=fi$ and $j\beta=fj$. 
The following commutative square is a pullback:
$$\SelectTips{cm}{}
\xymatrix{
F\cap G \ar[r]^-u \ar[d]_k & G \ar[d]^j \\
F \ar[r]_i & M }$$ 
We have $i\alpha k=fik=fju=j\beta u$. By the pullback property, there exists a unique morphism 
$\gamma:F\cap G\to F\cap G$ such that $k\gamma=\alpha k$ and $u\gamma=\beta u$.
We have $fju=j\beta u=ju\gamma$. Hence $ju:F\cap G\to M$ is fully invariant.

(ii) Let $h:G\to G$ be a morphism in $\A$. By hypothesis, $h$ can be extended to a morphism $f: M\to M$.
Hence $fj=jh$. Since $i$ is fully invariant, there exists a morphism $\alpha:F\to F$ such that $i\alpha=fi$.
Consider the pullback square from the proof of (i). We have $i\alpha k=fik=fju=jhu$. By the
pullback property, there exists a unique morphism $\gamma:F\cap G\to F\cap G$ such that $k\gamma=\alpha k$ and $u\gamma=hu$.
This shows that $u$ is fully invariant.
\end{proof}

\begin{cor} \label{c:Fds} Let $\A$ be an abelian category. Let $0\to F\to M\to M/F\to 0$
be a fully invariant short exact sequence in $\A$ and $M=M_1\oplus M_2$. Then:
\begin{enumerate}
\item The inclusion monomorphism $u:F\cap M_1\to M_1$ is fully invariant.
\item The induced epimorphism $q:M_1\to M/(F+M_2)$ is fully coinvariant.
\end{enumerate}
\end{cor}

\begin{proof} (1) Note that every morphism $M_1\to M_1$ can be extended to a morphism $M\to M$ in $\A$, 
and use Proposition \ref{p:FG2} (1) (ii).
\end{proof}

\begin{prop} \label{p:POPB} Let $\A$ be an abelian category.
\begin{enumerate}
\item Let $g:B\to B'$ be a cokernel and $d':B'\to C$ a morphism such that
$d=d'g:B\to C$ is a fully coinvariant retraction. Then $d'$ is a fully coinvariant retraction.
\item Let $g:B'\to B$ be a kernel and $i':A\to B'$ a morphism such that
$i=gi':A\to B$ is a fully invariant section. Then $i'$ is a fully
invariant section.
\end{enumerate}
\end{prop}

\begin{proof} (1) Note that $d'$ is a retraction. One may construct the following commutative diagram:
$$\SelectTips{cm}{}
\xymatrix{
A \ar[r]^-{i} \ar[d]_f & B \ar[r]^-{d} \ar[d]^g & C \ar@{=}[d] \\
A' \ar[r]_-{i'} & B' \ar[r]_{d'} & C }$$
where the rows are exact and the left square is a
pullback. Then $f$ is a cokernel and $i$ is a fully invariant
section. Now let $h:B'\to B'$ be a morphism. There exists a
monomorphism $s:C\to B$ such that $ds=1_C$. Consider the morphism
$sd'hg:B\to B$. Since $i$ is a fully invariant section, there
exists a morphism $\alpha:A\to A$ such that $sd'hgi=i\alpha$. Then
$d'hi'f=d'hgi=dsd'hgi=di\alpha=0$, whence we have $d'hi'=0$,
because $f$ is an epimorphism. Then there exists a morphism
$\gamma:C\to C$ such that $d'h=\gamma d'$. Hence $d'$ is a fully
coinvariant retraction.
\end{proof}

\begin{prop} \label{p:fimatrix} Let $\A$ be an abelian category.
\begin{enumerate}
\item Let $\left[\begin{smallmatrix} i&0\\ 0&i' \end{smallmatrix}\right]:A\oplus A'\to B\oplus B'$
be a fully invariant kernel in $\A$. Then $i:A\to B$ and $i':A'\to B'$ are fully invariant kernels.
\item Let $\left[\begin{smallmatrix} d&0\\ 0&d' \end{smallmatrix}\right]:B\oplus B'\to C\oplus C'$
be a fully coinvariant cokernel in $\A$. Then $d:B\to C$ and $d':B'\to C'$ are fully coinvariant cokernels.
\end{enumerate}
\end{prop}

\begin{proof} (1) Clearly, $i$ is a kernel. Let $h:B\to B$ be a morphism.
For the morphism $\left[\begin{smallmatrix} h&0\\0&1 \end{smallmatrix}\right]:B\oplus B'\to B\oplus B'$
there exists a morphism $\left[\begin{smallmatrix} a&b\\c&d \end{smallmatrix}\right]:A\oplus A'\to A\oplus A'$
such that \[\left[\begin{smallmatrix} h&0\\0&1 \end{smallmatrix}\right]\left[\begin{smallmatrix} i&0\\0&i' \end{smallmatrix}\right]=
\left[\begin{smallmatrix} i&0\\0&i' \end{smallmatrix}\right]\left[\begin{smallmatrix} a&b\\c&d \end{smallmatrix}\right].\]
Then $hi=ia$, which shows that $i$ is fully invariant.
\end{proof}

\begin{lem} \label{l:dsum} Let $\A$ be an abelian category with coproducts. 
Let $(M_k)_{k\in K}$ be a family of objects of $\A$ such that ${\rm Hom}_{\A}(M_k, M_l)=0$ 
for every $k,l\in K$ with $k\neq l$. 
Then the short exact sequences $0\to F_k\to M_k\to C_k\to 0$ 
are fully invariant for every $k\in K$ if and only if the induced short exact sequence
$0\to \bigoplus_{k\in K}F_k\to \bigoplus_{k\in K}M_k\to \bigoplus_{k\in K}C_k\to 0$ is fully invariant.
\end{lem}

\begin{proof} Straightforward.
\end{proof}

Let $\A$ be an abelian category. Recall that a \emph{preradical} $r$ of $\A$ is a subfunctor of the identity functor on $\A$, 
that is, $r:\A\to \A$ is a functor which assigns to each object $A$ of $\A$ a subobject $r(A)$ such that 
every morphism $A\to B$ induces a morphism $r(A)\to r(B)$ by restriction (e.g., see \cite[I.1]{BKN}). 

The following proposition relates fully invariant short exact sequences and preradicals, 
and will be implicitly used, without further reference. It is an immediate generalization of 
\cite[Proposition~I.6.2]{BKN} from module categories to abelian categories.

\begin{prop} \label{p:fiprerad} 
Let $\A$ be an abelian category. Then a short exact sequence $0\to F\to M\to C\to 0$ in $\A$
is fully invariant if and only if there is a preradical $r$ of $\A$ such that $r(M)=F$. 
\end{prop}

\section{$F$-split objects}

In this section we introduce our main concepts of (strongly) relative $F$-split 
and dual (strongly) relative $F$-split objects in abelian categories. 

\begin{defi} \label{d:fsplit} \rm Let $\mathcal{A}$ be an abelian category. Let $M$ be an object of $\A$, 
and let $0\to F\stackrel{i}\to N\stackrel{d}\to C\to 0$ be a fully invariant short exact sequence in $\mathcal{A}$.
Then $N$ is called:
\begin{enumerate}
\item \emph{(strongly) $M$-$F$-split} if for every morphism $g:M\to N$ the morphism $j:P\to M$
from the following pullback square is a (fully invariant) section:
$$\SelectTips{cm}{}
\xymatrix{
P \ar[r]^j \ar[d]_f & M \ar[d]^g \\
F \ar[r]_i & N }$$ 
\item \emph{dual (strongly) $M$-$F$-split} if for every morphism $g:N\to M$ the morphism $p:M\to Q$
from the following pushout square is a (fully coinvariant) retraction:
$$\SelectTips{cm}{}
\xymatrix{
N \ar[r]^d \ar[d]_g & C \ar[d]^h \\
M \ar[r]_p & Q }$$ 
\item \emph{(strongly) self-$F$-split} if $N$ is (strongly) $N$-$F$-split.
\item \emph{dual (strongly) self-$F$-split} if $N$ is dual (strongly) $N$-$F$-split.
\end{enumerate}
\end{defi}

\begin{lem} \label{l:kc} Let $\mathcal{A}$ be an abelian category. Let $M$ be an object of $\A$, 
and let $0\to F\stackrel{i}\to N\stackrel{d}\to C\to 0$ be a fully invariant short exact sequence in $\mathcal{A}$. Then:
\begin{enumerate}
\item $N$ is (strongly) $M$-$F$-split if and only if for every morphism $g:M\to N$, ${\rm ker}(dg)$ is a (fully invariant) section.
\item $N$ is dual (strongly) $M$-$F$-split if and only if for every morphism $g:N\to M$, ${\rm coker}(gi)$ 
is a (fully coinvariant) retraction.
\end{enumerate} 
\end{lem}

\begin{proof} (1) Consider the following diagram:
$$\SelectTips{cm}{}
\xymatrix{
P \ar[r]^j \ar@{-->}[d]_f & M \ar[d]^g \ar[r]^{dg} & C \ar@{=}[d] \\
F \ar[r]_i & N \ar[r]_d & C
}$$
Since $i={\rm ker}(d)$, the diagram may be completed to a pullback square $PMFN$ if and only if 
$j={\rm ker}(dg)$ \cite[Proposition~13.2]{M}. Now the conclusion is clear. 
\end{proof}

\begin{remark} \rm Consider the module category ${\rm Mod}(R)$, and let $M$, $N$ be right $R$-modules. 
We refer to notation and diagrams from Definition \ref{d:fsplit}.

(1) Let $F$ be a fully invariant submodule of $N$. Then $N$ is
(strongly) $M$-$F$-split if and only if for every homomorphism
$g:M\to N$, $j:P\to M$ is a (fully invariant) section if and only if for
every homomorphism $g:M\to N$, $P=g^{-1}(F)$ is a  (fully
invariant) direct summand of $M$.

(2) Let $F$ be a fully invariant submodule of $M$, or equivalently, let $C=M/F$ be a fully coinvariant factor module of $M$.
Then $N$ is dual (strongly) $M$-$F$-split if and only if for every homomorphism $g:N\to M$,
$p:M\to Q$ is a (fully coinvariant) retraction if and only if
for every homomorphism $g:N\to M$, ${\rm ker}(p)$ is a (fully invariant) section if and only if 
for every homomorphism $g:N\to M$, $gi$ is a (fully invariant) section if and only if 
for every homomorphism $g:N\to M$, $g(F)$ is a (fully invariant) direct summand of $M$.
\end{remark}

\begin{exam} \rm (1) Let $M$ and $N$ be objects of an abelian category $\mathcal{A}$. 
Obviously, $N$ is strongly $M$-$N$-split and $N$ is dual strongly $M$-$0$-split.

(2) Consider the module category ${\rm Mod}(R)$. Let $M$ and $N$ be right $R$-modules. 
Then $N$ is $M$-$F$-split if and only if $N$ is $M$-$F$-inverse split in the sense of \cite{Ungor}. 
For $F=Z^2_M(N)$ (see the notation from the last section of our paper), a module $N$ is $M$-$F$-split 
if and only if $N$ is $M$-$T$-Rickart in the sense of \cite{Atani12}.

Let us note that our categorical dual notion of dual (relative) $F$-splitness does not 
coincide with dual (relative) $F$-inverse splitness in the sense of \cite{Ungor2}
(e.g., apart from their definitions, compare our forthcoming Theorem \ref{t:key} and \cite[Theorem~2.2]{Ungor2}). 
But for $F={\overline Z}^2_M(N)$ (see the notation from the last section of our paper), 
a module $N$ is dual $M$-$F$-split if and only if $N$ is $M$-$T$-dual Rickart in the sense of \cite{Atani16}.
\end{exam}

Strong self-$F$-splitness and self-$F$-splitness are related by the following result. 

\begin{theor} \label{t:rel} Let $\A$ be an abelian category.
Let $0\to F\stackrel{i}\to M\stackrel{d}\to C\to 0$ be a fully invariant short exact sequence in $\A$.
\begin{enumerate}
\item The following are equivalent:
\begin{enumerate}[(i)]
\item $M$ is strongly self-$F$-split.
\item $M$ is self-$F$-split and every direct summand of $M$ which contains $F$ is fully invariant.
\end{enumerate}
\item The following are equivalent:
\begin{enumerate}[(i)]
\item $M$ is dual strongly self-$F$-split.
\item $M$ is dual self-$F$-split and every direct summand of $M$ which is contained in $F$ is fully invariant.
\end{enumerate}
\end{enumerate}
\end{theor}

\begin{proof} (1) (i) $\Rightarrow$ (ii) Assume that $M$ is strongly self-$F$-split.
Clearly, $M$ is self-$F$-split. Now let $X$ be a direct summand of $M$ such that $F\subseteq X$.
Write $M=X\oplus Y$ for some subobject $Y$ of $M$, and denote by $u:F\to X$ the inclusion monomorphism.
Then $i=\left[\begin{smallmatrix} u\\ 0 \end{smallmatrix}\right]:F\to X\oplus Y$.
We claim that the following commutative square is a pullback:
$$\SelectTips{cm}{}
\xymatrix{
X \ar[r]^-{\left[\begin{smallmatrix} 1\\0 \end{smallmatrix}\right]} \ar[d]_0 &
X\oplus Y \ar[d]^{\left[\begin{smallmatrix} 0&0\\ 0&1 \end{smallmatrix}\right]} \\
F \ar[r]_-{\left[\begin{smallmatrix} u\\ 0 \end{smallmatrix}\right]} & X\oplus Y }$$
To this end, let $\alpha:Z\to F$ and
$\left[\begin{smallmatrix} \beta_1\\ \beta_2 \end{smallmatrix}\right]:Z\to X\oplus Y$ be morphisms
such that $\left[\begin{smallmatrix} u\\0 \end{smallmatrix}\right]\alpha=
\left[\begin{smallmatrix} 0&0\\0&1 \end{smallmatrix}\right]
\left[\begin{smallmatrix} \beta_1\\ \beta_2 \end{smallmatrix}\right]$.
Then $\beta_2=0$ and $u\alpha=0$. Hence $\alpha=0$, because $u$ is a monomorphism. It is easy to check that
$\beta_1:Z\to X$ is the unique morphism such that $0\beta_1=\alpha$ and
$\left[\begin{smallmatrix} 1\\0 \end{smallmatrix}\right]\beta_1=
\left[\begin{smallmatrix} \beta_1\\ \beta_2 \end{smallmatrix}\right]$.
Hence the required square is a pullback.

Since $M$ is strongly self-$F$-split, it follows that the upper horizontal morphism is a fully invariant section,
hence $X$ is a fully invariant direct summand of $M$.

(ii) $\Rightarrow$ (i) Assume that (ii) holds. Consider a pullback square:
$$\SelectTips{cm}{}
\xymatrix{
P \ar[r]^j \ar[d]_f & M \ar[d]^g \\
F \ar[r]_i & M }$$
Since $M$ is self-$F$-split, $P$ is a direct summand of $M$.
Since $i$ is fully invariant, there is a morphism $\alpha:F\to F$ such that $gi=i\alpha$.
The pullback property yields a unique morphism $\gamma:F\to P$ such that $f\gamma=\alpha$ and $j\gamma=i$.
Then $\gamma$ is a monomorphism, hence $F\subseteq P$. By hypothesis, $P$ must be a fully invariant direct summand of $M$.
Hence $M$ is strongly self-$F$-split.
\end{proof}

The following proposition generalizes \cite[Proposition~2.17]{Ungor}.

\begin{prop} Let $\A$ be an abelian category. Let $0\to F\stackrel{i}\to M\stackrel{d}\to C\to 0$
be a fully invariant short exact sequence in $\A$.
\begin{enumerate}
\item Assume that $M$ is (strongly) self-$F$-split. Let $N$ be a direct summand of $M$ with $F\subseteq N$.
Then for every (fully invariant) direct summand $K$ of $M$, $K\cap N$ is a (fully invariant) direct summand of $M$.
\item Assume that $M$ is dual (strongly) self-$F$-split. Let $N$ be a direct summand of $M$ with $N\subseteq F$.
Then for every (fully invariant) direct summand $K$ of $M$, $K+N$ is a (fully invariant) direct summand of $M$.
\end{enumerate}
\end{prop}

\begin{proof} (1) Denote by $u:F\to N$ the inclusion monomorphism. Write $M=N\oplus X=K\oplus Y$
for some subobjects $X,Y$ of $M$. As in the proof of Theorem \ref{t:rel}, the following commutative square is a pullback:
$$\SelectTips{cm}{}
\xymatrix{
N \ar[rr]^-{j=\left[\begin{smallmatrix} 1\\0 \end{smallmatrix}\right]} \ar[d]_0 &&
M=N\oplus X \ar[d]^{f_1=\left[\begin{smallmatrix} 0&0\\0&1 \end{smallmatrix}\right]} \\
F \ar[rr]_-{i=\left[\begin{smallmatrix} u\\0 \end{smallmatrix}\right]} && M=N\oplus X }$$

Let $k:K\cap N\to K$ and $n:K\cap N\to N$ be the inclusion monomorphisms. Consider the following commutative diagram:
$$\SelectTips{cm}{}
\xymatrix{
(K\cap N)\oplus Y \ar[d]_{\left[\begin{smallmatrix} 1&0 \end{smallmatrix}\right]}
\ar[rr]^{l=\left[\begin{smallmatrix} k&0\\0&0 \end{smallmatrix}\right]} &&
M=K\oplus Y \ar[d]^{\left[\begin{smallmatrix} 1&0 \end{smallmatrix}\right]} \\
K\cap N \ar[rr]^-k \ar[d]_n && K \ar[d]^{\left[\begin{smallmatrix} 1\\0 \end{smallmatrix}\right]} \\
N \ar[rr]_-{j=\left[\begin{smallmatrix} j_1\\j_2 \end{smallmatrix}\right]} && M=K\oplus Y }$$
The upper and the lower squares are clearly pullbacks, hence so is the outer rectangle \cite[Lemma~5.1]{Kelly}.
Denote $f_2=\left[\begin{smallmatrix} 1&0\\0&0 \end{smallmatrix}\right]:K\oplus Y\to K\oplus Y$.
Glueing together the above 3 pullbacks, one obtains the following pullback square \cite[Lemma~5.1]{Kelly}:
$$\SelectTips{cm}{}
\xymatrix{
(K\cap N)\oplus Y \ar[r]^-{l} \ar[d]_0 & M \ar[d]^{f_1f_2} \\
F \ar[r]_-{i} & M }$$
Since $M$ is (strongly) self-$F$-split, $l=\left[\begin{smallmatrix} k&0\\0&0 \end{smallmatrix}\right]$
is a (fully invariant) section. It follows that $K\cap N$ is a
(fully invariant) direct summand of $K$ (by Proposition \ref{p:fimatrix}),
and so a (fully invariant) direct summand of $M$ (by Lemma \ref{l:trans}).
\end{proof}

Let $\mathcal{A}$ be an abelian category. We say that an object $M$ of $\mathcal{A}$ has the 
\emph{(strong) summand intersection property}, briefly \emph{SIP (SSIP)}, 
for a class $\mathcal{C}$ of direct summands of $M$ if 
the intersection of any finite family (any family) of objects from $\mathcal{C}$ belongs to $\mathcal{C}$.  
Dually, an object $M$ of $\mathcal{A}$ has the \emph{(strong) summand sum property}, briefly \emph{SSP (SSSP)}, 
for a class $\mathcal{C}$ of direct summands of $M$ if the sum of any finite family (any family) of objects from 
$\mathcal{C}$ belongs to $\mathcal{C}$. 

\begin{cor} \label{c:SIP} Let $\A$ be an abelian category. Let $0\to F\to M\to C\to 0$
be a fully invariant short exact sequence in $\A$.
\begin{enumerate}
\item Assume that $M$ is (strongly) self-$F$-split. Then:
\begin{enumerate}[(i)]
\item For every (fully invariant) direct summand $K$ of $M$, $F\cap K$ is a (fully invariant) direct summand of $M$.
\item $M$ has SIP for (fully invariant) direct summands containing $F$.
\end{enumerate}
\item Assume that $M$ is dual (strongly) self-$F$-split. Then:
\begin{enumerate}[(i)]
\item For every (fully invariant) direct summand $K$ of $M$, $F+K$ is a (fully invariant) direct summand of $M$.
\item $M$ has SSP for (fully invariant) direct summands contained in $F$.
\end{enumerate}
\end{enumerate}
\end{cor}

\begin{theor} \label{t:epimono} Let $\A$ be an abelian category.
Let $0\to F\stackrel{i}\to N\stackrel{d}\to C\to 0$ be a fully invariant short exact sequence in $\A$.
Let $e:M\to M'$ be an epimorphism and let $m:N'\to N$ be a monomorphism in $\A$.
\begin{enumerate}
\item Assume that the inclusion monomorphism $u:F\cap N'\to N'$ is fully invariant and
$N$ is (strongly) $M$-$F$-split. Then $N'$ is (strongly) $M'$-$(F\cap N')$-split.
\item Assume that the induced epimorphism $q:N/N'\to ((F+N')/N')$ is fully coinvariant and
$N$ is dual (strongly) $M$-$F$-split. Then $N/N'$ is dual (strongly) $M/M'$-$((F+N')/N')$-split.
\end{enumerate}
\end{theor}

\begin{proof} (1) Let $g:M'\to N'$ be a morphism in $\A$. Let $G=F\cap N'$ and consider the pullback of $u$ and $g$ to get
morphisms $l:Q\to M'$ and $q:Q\to G$. Let $t:G\to F$ be the inclusion monomorphism. Consider the pullback of $i$ and $mge$ to get
morphisms $j:P\to M$ and $f:P\to F$. The pullback property of the square $GN'FN$ yields a unique morphism
$h:P\to G$ such that $th=f$ and $uh=gej$. The pullback property of the square $QM'GN'$ yields a unique morphism
$p:P\to Q$ such that $qp=h$ and $lp=ej$. In this way one constructs the following commutative diagram:
$$\SelectTips{cm}{}
\xymatrix{
P \ar[r]^j \ar[d]^p \ar@/_1pc/[dd]_h \ar@/_3pc/[ddd]_f & M \ar[d]^e \\
Q \ar[r]^l \ar[d]^q & M' \ar[d]^g \\
G \ar[r]^u \ar[d]_t & N' \ar[d]^m \\
F \ar[r]_i & N }
$$
The rectangle $PMFN$ and the square $GN'FN$ are pullbacks, hence so is the rectangle $PMGN'$ \cite[Lemma~5.1]{Kelly}.
Since the square $QM'GN'$ is a pullback, so is the square $PMQM'$ \cite[Lemma~5.1]{Kelly}.
Since $N$ is (strongly) $M$-$F$-split, $j$ is a (fully invariant) section.
It is easy to check that the square $PMQM'$ is also a pushout, hence $l$ is a section.
We may construct the following commutative diagram with exact rows:
$$\SelectTips{cm}{}
\xymatrix{
P \ar[r]^j \ar[d]_p & M \ar[r]^b \ar[d]^e & D \ar@{=}[d] \\
Q \ar[r]_l & M' \ar[r]_{d'} & D }
$$
If $j$ is a fully invariant section, then $b=d'e$ is a fully coinvariant retraction.
Hence $d'$ is a fully coinvariant retraction by Proposition \ref{p:POPB},
and so $l$ is a fully invariant section. Hence $N'$ is (strongly) $M'$-$(F\cap N')$-split.
\end{proof}

The following corollary generalizes \cite[Proposition~2.12 and Theorem~4.2]{Ungor}.

\begin{cor} \label{c:strel} Let $\A$ be an abelian category. Let $M$ be an object 
and let $0\to F\to N\to C\to 0$ be a fully invariant short exact sequence in $\A$. Then:
\begin{enumerate}
\item $N$ is (strongly) $M$-$F$-split if and only if for every direct summand $M_1$ of $M$ and 
for every direct summand $N_1$ of $N$, $N_1$ is (strongly) $M_1$-$(F\cap N_1)$-split.
\item $N$ is dual (strongly) $M$-$F$-split if and only if for every direct summand $M_1$ of $M$ 
and for every direct summand $N_1$ of $N$, $N/N_1$ is dual (strongly) $M/M_1$-$((F+N_1)/N_1)$-split.
\end{enumerate}
\end{cor}

\begin{proof} (1) Note that the inclusion monomorphism $u:F\cap N_1\to N_1$ 
is fully invariant by Corollary \ref{c:Fds}. Then use Theorem \ref{t:epimono}.
\end{proof}

The following corollary generalizes \cite[Lemma~2.10]{Ungor}.

\begin{cor} Let $\A$ be an abelian category. Let $0\to F\to M\to M/F\to 0$
be a fully invariant short exact sequence and $0\to N\to M\to M/N\to 0$
a short exact sequence in $\A$.
\begin{enumerate}
\item Assume that every morphism $N\to N$ can be extended
to a morphism $M\to M$ in $\A$.  If $M$ is (strongly)
self-$F$-split, then $N$ is (strongly) self-$(F\cap N)$-split.
\item Assume that every morphism $M/N\to M/N$ can be
lifted to a morphism $M\to M$ in $\A$. If $M$ is dual
(strongly) self-$F$-split, then $M/N$ is dual (strongly)
self-$((F+N)/N)$-split.
\end{enumerate}
\end{cor}

\begin{proof} By Proposition \ref{p:FG2}, the inclusion
monomorphism $u: F\cap N\to N$ is fully invariant. Then use Theorem \ref{t:epimono}.
\end{proof}

Let $\A$ be an abelian category and let $r$ be a preradical of $\A$. Then one may define a 
preradical $r^{-1}$ of $\A^{\rm op}$ by $r^{-1}(A)=A/r(A)$ for every object $A$ of $\A$. 
Recall that $r$ is called \emph{hereditary} if $r$ is a left exact functor, 
and \emph{cohereditary} if $r^{-1}$ is a right exact functor (e.g., see \cite{BKN}).

\begin{cor} Let $\A$ be an abelian category and let $r$ be a preradical of $\A$.
Let $e:M\to M'$ be an epimorphism and let $m:N'\to N$ be a monomorphism in $\A$.
\begin{enumerate}
\item Assume that $r$ is hereditary and $N$ is (strongly) $M$-$r(N)$-split. 
Then $N'$ is (strongly) $M'$-$r(N')$-split.
\item Assume that $r$ is cohereditary and $N$ is dual (strongly) $M$-$r(N)$-split. 
Then $N/N'$ is dual (strongly) $M/M'$-$r(N/N')$-split.
\end{enumerate}
\end{cor}

\begin{proof} If $r$ is hereditary, then $r(N)\cap N'=r(N')$ is fully invariant in $N'$ \cite[I.2.1]{BKN}. 
Also, if $r$ is cohereditary, then $(r(N)+N')/N'=r(N/N')$ is fully invariant in $N/N'$ \cite[I.2.8]{BKN}. 
Then use Theorem \ref{t:epimono}. 
\end{proof}

\section{$F$-split objects versus Rickart objects}

Let us note that (strongly) relative $F$-split and dual (strongly) relative $F$-split objects 
generalize (strongly) relative Rickart and dual (strongly) relative Rickart objects defined as follows.

\begin{defi}[{\cite{C-K,CO1}}] \rm Let $M$ and $N$ be objects of an abelian category $\mathcal{A}$. Then $N$ is called:
\begin{enumerate}
\item \emph{(strongly) $M$-Rickart} if for every morphism $f:M\to N$, ${\rm ker}(f)$ is a (fully invariant) section.
\item \emph{dual (strongly) $M$-Rickart} if for every morphism $f:M\to N$, ${\rm coker}(f)$ is a (fully invariant) retraction.
\item \emph{(strongly) self-Rickart} if $N$ is (strongly) $N$-Rickart.
\item \emph{dual (strongly) self-Rickart} if $N$ is dual (strongly) $N$-Rickart.
\end{enumerate}
\end{defi}

\begin{remark} \label{r:Rickart} \rm Let $M$ and $N$ be objects of an abelian category $\mathcal{A}$. Then
$N$ is (strongly) $M$-Rickart if and only if $N$ is (strongly) $M$-$0$-split, while
$N$ is (strongly) dual $M$-Rickart if and only if $N$ is dual (strongly) $M$-$M$-split. 
Also, $N$ is (strongly) self-Rickart if and only if $N$ is (strongly) self-$0$-split, while
$N$ is (strongly) dual self-Rickart if and only if $N$ is dual (strongly) self-$N$-split. 
\end{remark}

The following theorem generalizes \cite[Theorem~2.3 and Proposition~2.4]{Ungor}, and is one of the key results of the paper.

\begin{theor}\label{t:key} Let $\A$ be an abelian category.
Let $0\to F\stackrel{i}\to M\stackrel{d}\to C\to 0$ be a fully invariant short exact sequence in $\A$.
\begin{enumerate}
\item The following are equivalent:
\begin{enumerate}[(i)]
\item $M$ is (strongly) self-$F$-split.
\item $M\cong F\oplus C$ and $C$ is (strongly) self-Rickart.
\end{enumerate}
\item The following are equivalent:
\begin{enumerate}[(i)]
\item $M$ is dual (strongly) self-$F$-split.
\item $M\cong F\oplus C$ and $F$ is dual (strongly) self-Rickart.
\end{enumerate}
\end{enumerate}
\end{theor}

\begin{proof}
(1) (i) $\Rightarrow$ (ii) Assume that $M$ is (strongly) self-$F$-split.
Then the morphism $j:P\to M$ from the following pullback square is a (fully invariant) section:
$$\SelectTips{cm}{}
\xymatrix{
P \ar[r]^j \ar[d]_f & M \ar[d]^{1_M} \\
F \ar[r]_i & M }$$
Since $f$ must be an isomorphism, it follows that $i$ is a section. Hence $M\cong F\oplus C$.

Let $g:C\to C$ be a morphism in $\A$ with kernel $k:K\to C$.
Consider the following commutative square:
$$\SelectTips{cm}{}
\xymatrix{
F\oplus K \ar[r]^-{\left[\begin{smallmatrix} 1&0\\0&k \end{smallmatrix}\right]}
\ar[d]_{\left[\begin{smallmatrix} 1&0 \end{smallmatrix}\right]} &
F\oplus C \ar[d]^{\left[\begin{smallmatrix} 1&0\\0&g \end{smallmatrix}\right]}\\
F \ar[r]_-{\left[\begin{smallmatrix} 1\\0 \end{smallmatrix}\right]} & F\oplus C }$$
We claim that it is a pullback square. To this end, let $\alpha:Z\to F$ and
${\left[\begin{smallmatrix} \beta_1\\ \beta_2 \end{smallmatrix}\right]}:Z\to F\oplus C$ be morphisms in $\A$
such that ${\left[\begin{smallmatrix} 1&0\\0&g \end{smallmatrix}\right]}
{\left[\begin{smallmatrix} \beta_1\\ \beta_2 \end{smallmatrix}\right]}=
{\left[\begin{smallmatrix} 1\\0 \end{smallmatrix}\right]}\alpha$,
that is, $\beta_1=\alpha$ and $g\beta_2=0$. Then there exists a unique morphism $\gamma:Z\to K$ such that $\beta_2=k\gamma$.
Consider the morphism ${\left[\begin{smallmatrix} \beta_1 \\ \gamma \end{smallmatrix}\right]}:Z\to F\oplus K$.
Then $\left[\begin{smallmatrix} 1&0 \end{smallmatrix}\right]\left[\begin{smallmatrix} \beta_1 \\ \gamma \end{smallmatrix}\right]=
\alpha$ and $\left[\begin{smallmatrix} 1&0\\0&k \end{smallmatrix}\right]
\left[\begin{smallmatrix} \beta_1 \\ \gamma \end{smallmatrix}\right]=
\left[\begin{smallmatrix} \beta_1\\ \beta_2 \end{smallmatrix}\right]$.
If $\left[\begin{smallmatrix} \gamma_1\\ \gamma_2 \end{smallmatrix}\right]:Z\to F\oplus K$ is another morphism
such that $\left[\begin{smallmatrix} 1&0 \end{smallmatrix}\right]
\left[\begin{smallmatrix} \gamma_1 \\ \gamma_2 \end{smallmatrix}\right]=
\alpha$ and $\left[\begin{smallmatrix} 1&0\\0&k \end{smallmatrix}\right]
\left[\begin{smallmatrix} \gamma_1 \\ \gamma_2 \end{smallmatrix}\right]=
\left[\begin{smallmatrix} \beta_1\\ \beta_2 \end{smallmatrix}\right]$,
then $\gamma_1=\alpha=\beta_1$ and $k\gamma_2=\beta_2=k\gamma$.
This implies that $\gamma_2=\gamma$, because $k$ is a monomorphism, and so
$\left[\begin{smallmatrix} \gamma_1 \\ \gamma_2 \end{smallmatrix}\right]=
\left[\begin{smallmatrix} \beta_1 \\ \gamma \end{smallmatrix}\right]$.
Hence the square is a pullback.

Since $M\cong F\oplus C$ is (strongly) self-$F$-split,
$\left[\begin{smallmatrix} 1&0\\0&k \end{smallmatrix}\right]$ is a
(fully invariant) section. It follows that $k$ is a (fully invariant) section (by Proposition \ref{p:fimatrix}).
Hence $C$ is (strongly) self-Rickart.

(ii)$\Rightarrow$ (i) Assume that $M\cong F\oplus C$ and $C$ is (strongly) self-Rickart.
Let $f=\left[\begin{smallmatrix} a&b\\c&d \end{smallmatrix}\right]:F\oplus C\to F\oplus C$ be a morphism.
Since $c:F\to C$ and $F$ is a fully invariant subobject of $M\cong F\oplus C$, it follows that $c=0$, hence
$f=\left[\begin{smallmatrix} a&b\\0&d \end{smallmatrix}\right]$. Denote $k={\rm ker}(d):K\to C$.
We have $\left[\begin{smallmatrix} a&b\\0&d \end{smallmatrix}\right]
\left[\begin{smallmatrix} 1&0\\0&k \end{smallmatrix}\right]=
\left[\begin{smallmatrix} 1\\0 \end{smallmatrix}\right]\left[\begin{smallmatrix} a&bk \end{smallmatrix}\right]$,
hence the following square is commutative:
$$\SelectTips{cm}{}
\xymatrix{
F\oplus K \ar[r]^-{\left[\begin{smallmatrix} 1&0\\0&k \end{smallmatrix}\right]}
\ar[d]_{\left[\begin{smallmatrix} a&bk \end{smallmatrix}\right]} &
F\oplus C \ar[d]^{\left[\begin{smallmatrix} a&b\\0&d \end{smallmatrix}\right]}\\
F \ar[r]_-{\left[\begin{smallmatrix} 1\\0 \end{smallmatrix}\right]} & F\oplus C }$$
We claim that it is a pullback square. To this end, let $\alpha:Z\to F$ and
${\left[\begin{smallmatrix} \beta_1\\ \beta_2 \end{smallmatrix}\right]}:Z\to F\oplus C$ be morphisms in $\A$
such that ${\left[\begin{smallmatrix} a&b\\0&d \end{smallmatrix}\right]}
{\left[\begin{smallmatrix} \beta_1\\ \beta_2 \end{smallmatrix}\right]}=
{\left[\begin{smallmatrix} 1\\0 \end{smallmatrix}\right]}\alpha$,
that is, $a\beta_1+b\beta_2=\alpha$ and $d\beta_2=0$. Then there exists a unique morphism
$\gamma:Z\to K$ such that $\beta_2=k\gamma$.
Consider the morphism ${\left[\begin{smallmatrix} \beta_1 \\ \gamma \end{smallmatrix}\right]}:Z\to F\oplus K$.
Then $\left[\begin{smallmatrix} a&bk \end{smallmatrix}\right]\left[\begin{smallmatrix} \beta_1 \\ \gamma \end{smallmatrix}\right]=
a\beta_1+bk\gamma=a\beta_1+b\beta_2=\alpha$ and $\left[\begin{smallmatrix} 1&0\\0&k \end{smallmatrix}\right]
\left[\begin{smallmatrix} \beta_1 \\ \gamma \end{smallmatrix}\right]=
\left[\begin{smallmatrix} \beta_1\\ k\gamma \end{smallmatrix}\right]=
\left[\begin{smallmatrix} \beta_1\\ \beta_2 \end{smallmatrix}\right]$.
If $\left[\begin{smallmatrix} \gamma_1\\ \gamma_2 \end{smallmatrix}\right]:Z\to F\oplus K$ is another morphism
such that $\left[\begin{smallmatrix} a&bk \end{smallmatrix}\right]
\left[\begin{smallmatrix} \gamma_1 \\ \gamma_2 \end{smallmatrix}\right]=
\alpha$ and $\left[\begin{smallmatrix} 1&0\\0&k \end{smallmatrix}\right]
\left[\begin{smallmatrix} \gamma_1 \\ \gamma_2 \end{smallmatrix}\right]=
\left[\begin{smallmatrix} \beta_1\\ \beta_2 \end{smallmatrix}\right]$,
then $\gamma_1=\beta_1$ and $k\gamma_2=\beta_2=k\gamma$.
This implies that $\gamma_2=\gamma$, because $k$ is a monomorphism, and so
$\left[\begin{smallmatrix} \gamma_1 \\ \gamma_2 \end{smallmatrix}\right]=
\left[\begin{smallmatrix} \beta_1 \\ \gamma \end{smallmatrix}\right]$.
This shows that the square is a pullback.

Since $C$ is (strongly) self-Rickart, $k$ is a (fully invariant) section. It follows that
$\left[\begin{smallmatrix} 1&0\\0&k \end{smallmatrix}\right]$ is a section.
Hence $M\cong F\oplus C$ is self-$F$-split.

If $k$ is fully invariant, then we claim that
$\left[\begin{smallmatrix} 1&0\\0&k \end{smallmatrix}\right]$ is also fully invariant.
To this end, let $\left[\begin{smallmatrix} h_1&h_2\\h_3&h_4 \end{smallmatrix}\right]:F\oplus C\to F\oplus C$ be a morphism.
As above, we must have $h_3=0$. For the morphism $h_4:C\to C$, there exists a morphism $\alpha:K\to K$ such that
$h_4k=k\alpha$. Consider the morphism $\left[\begin{smallmatrix} h_1&h_2k\\0&\alpha \end{smallmatrix}\right]:F\oplus K\to F\oplus K$.
Then $\left[\begin{smallmatrix} h_1&h_2\\0&h_4 \end{smallmatrix}\right]\left[\begin{smallmatrix} 1&0\\0&k \end{smallmatrix}\right]=
\left[\begin{smallmatrix} 1&0\\0&k \end{smallmatrix}\right]\left[\begin{smallmatrix} h_1&h_2k\\0&\alpha \end{smallmatrix}\right]$,
which shows that $\left[\begin{smallmatrix} 1&0\\0&k \end{smallmatrix}\right]$ is fully invariant.
Hence, if $C$ is strongly self-Rickart, then $M\cong F\oplus C$ is strongly self-$F$-split.
\end{proof}

\begin{cor} Let $M$ be an object of an abelian category $\A$. Then:
\begin{enumerate}
\item $M$ is (strongly) self-Rickart if and only if 
for every split fully invariant short exact sequence $0\to F\to M\to C\to 0$, $M$ is
(strongly) self-$F$-split.
\item $M$ is dual (strongly) self-Rickart if and only if 
for every split fully invariant short exact sequence $0\to F\to M\to C\to 0$, $M$
is dual (strongly) self-$F$-split.
\end{enumerate}
\end{cor}

\begin{proof} (1) For the direct implication, let $0\to F\to M\to C\to 0$ be a
split fully invariant short exact sequence. Then $M\cong F\oplus C$.
By \cite[Corollary~2.11]{C-K} (\cite[Corollary~2.18]{CO1}), $C$ is (strongly) self-Rickart.
Hence $M$ is (strongly) self-$F$-split by Theorem \ref{t:key}.

Conversely, for $F=0$, $M$ is (strongly) self-$0$-split, that is, $M$ is (strongly) self-Rickart.
\end{proof}

\begin{exam} \rm Consider the abelian group $G=\mathbb{Z}_{p^2}\oplus \mathbb{Z}_q$, where $p$ and $q$ are distinct primes
and we denote $\mathbb{Z}_n=\mathbb{Z}/n\mathbb{Z}$ for every natural number $n$. The subgroups of $G$ are 
$0$, $G$, $H_1$, $H_2$, $H_3$ and $H_4$, where the last 4 subgroups have orders $q$, $p$, $pq$ and $p^2$ respectively. 
Note that all subgroups of $G$ are fully invariant, while $0$, $G$, $H_1$ and $H_4$ are direct summands of $G$. 
Clearly, $G$ is self-$G$-split. But $G$ is not self-$0$-split, i.e., self-Rickart, because its direct summand 
$H_4\cong \mathbb{Z}_{p^2}$ is not self-Rickart \cite[Theorem~5.6]{LRR10}. 
By Theorem \ref{t:key}, $G$ cannot be either self-$H_2$-split or self-$H_3$-split. 
Again by Theorem \ref{t:key}, $G\cong H_1\oplus H_4$ is not self-$H_1$-split, because $H_4$ is not self-Rickart, 
while $G\cong H_1\oplus H_4$ is strongly self-$H_4$-split, because $H_1\cong \mathbb{Z}_q$ is strongly self-Rickart 
\cite[Corollary~3.9]{CO1}. Hence $G$ is strongly self-$F$-split for $F\in \{H_4,G\}$ and not self-$F$-split for 
$F\in \{0,H_1,H_2,H_3\}$. Similarly, $G$ is dual strongly self-$F$-split for $F\in \{H_1,G\}$ and 
not dual self-$F$-split for $F\in \{0,H_2,H_3,H_4\}$. This also shows that there are abelian groups 
which are self-$F$-split but not dual self-$F$-split, and viceversa.
\end{exam}

\begin{exam} \rm Consider the abelian group $G=\mathbb{Z}_p\oplus \mathbb{Z}_p\oplus \mathbb{Z}\oplus \mathbb{Q}$ 
for some prime $p$, and its subgroup $F\cong \mathbb{Z}_p\oplus \mathbb{Z}_p$. By \cite[Lemma~1.9]{OHS}, 
$F$ is a fully invariant subgroup of $G$, because 
$\Hom_{\mathbb{Z}}(\mathbb{Z}_p\oplus \mathbb{Z}_p,\mathbb{Z}\oplus \mathbb{Q})=0$. 
By Theorem \ref{t:key}, $G\cong F\oplus \mathbb{Z}\oplus \mathbb{Q}$ is self-$F$-split, because 
$\mathbb{Z}\oplus \mathbb{Q}$ is self-Rickart \cite[Example~2.15]{LRR12}. On the other hand, 
$G$ is not strongly self-$F$-split, because $\mathbb{Z}\oplus \mathbb{Q}$ is not strongly self-Rickart
\cite[Theorem~3.6]{CO1}. By Theorem \ref{t:key}, $G\cong F\oplus \mathbb{Z}\oplus \mathbb{Q}$ is dual self-$F$-split, 
because $F$ is clearly dual self-Rickart being semisimple. On the other hand, 
$G$ is not dual strongly self-$F$-split, because $F$ is not dual strongly self-Rickart \cite[Theorem~3.6]{CO1}.
\end{exam}

Let $\A$ be an abelian category and let $r$ be a preradical of $\A$. 
Recall that $r$ is called \emph{idempotent} if $rr=r$, and \emph{radical} if $rr^{-1}=0$. 
Note that every hereditary preradical is idempotent, and every cohereditary preradical is a radical (e.g., see \cite{BKN}).

\begin{exam} \rm For every abelian group $G$, denote by $t(G)$ the set of elements of $G$ having finite order (torsion elements), 
and by $d(G)$ the sum of its divisible (injective) subgroups. 
Then $t$ is a hereditary radical and $d$ is an idempotent radical of the category Ab of abelian groups.
Let $G$ be a finitely generated abelian group. Then there are two classical direct sum decompositions 
$G=t(G)\oplus F=d(G)\oplus R$ for some subgroups $F$ (torsionfree) and $R$ (reduced) of $G$. 
Note that $F\cong \mathbb{Z}^n$ for some $n\in \mathbb{N}$, hence $F$ is projective, 
while $d(G)$ is injective. It follows that $G$ is self-$t(G)$-split and dual self-$d(G)$-split 
by Theorem \ref{t:key} and \cite[Corollary~4.8]{C-K}. By \cite[Corollary~3.9]{CO1}, 
$G$ is strongly self-$t(G)$-split if and only if $F\cong \mathbb{Z}$, 
while $G$ is strongly dual self-$d(G)$-split if and only if $R\cong \mathbb{Z}_{p^{\infty}}$ 
(the Pr\"ufer $p$-group) for some prime $p$.  
\end{exam}

\begin{exam} \rm Consider the ring $R=\left[\begin{smallmatrix} \mathbb{Z} & \mathbb{Z} \\ 0 & \mathbb{Z} \end{smallmatrix}\right]$
and $M=R_R$. Then $F=\left[\begin{smallmatrix} \mathbb{Z}&\mathbb{Z} \\ 0&0\end{smallmatrix}\right]$ is an ideal of $R$, 
hence it is a fully invariant submodule of $M$. Moreover, we have $M=F\oplus C$, where 
$C=\left[\begin{smallmatrix} 0&0 \\ 0&\mathbb{Z} \end{smallmatrix}\right]$. 
Since $C\cong \mathbb{Z}$ is strongly self-Rickart \cite[Corollary~3.9]{CO1}, 
$M$ is strongly self-$F$-split by Theorem \ref{t:key}. Note that $F$ is self-Rickart \cite[Example~1.2]{LRR12}. 
If $F$ were dual self-Rickart, then it would be self-regular, 
and so ${\rm End}_R(F)\cong \mathbb{Z}$ would be a von Neumann regular ring, contradiction. 
Hence $F$ is not dual self-Rickart, and so $M$ is not dual self-$F$-split by Theorem \ref{t:key}. 
\end{exam}

\begin{exam} \rm Consider the ring $R=\left[\begin{smallmatrix} K & K \\ 0 & K \end{smallmatrix}\right]$ for some field $K$ 
and $M=R_R$. Then $F=\left[\begin{smallmatrix} K&K \\ 0&0\end{smallmatrix}\right]$ is an ideal of $R$, 
hence it is a fully invariant submodule of $M$. Moreover, we have $M=F\oplus C$, where 
$C=\left[\begin{smallmatrix} 0&0 \\ 0&K \end{smallmatrix}\right]$. 
Since $C\cong K$ is strongly self-Rickart and the indecomposable $F$ is strongly dual self-Rickart by \cite[Example~3.9]{LRR11}, 
it follows that $M$ is both strongly self-$F$-split and dual strongly self-$F$-split by Theorem \ref{t:key}. 
\end{exam}

Now we may give another result relating strong self-$F$-splitness and self-$F$-splitness. 
First recall that a ring $R$ is called \emph{abelian} if every idempotent element of $R$ is central. 

\begin{theor} \label{t:endab} Let $\A$ be an abelian category. Let $0\to F\to M\to C\to 0$
be a fully invariant short exact sequence in $\A$.
\begin{enumerate}
\item $M$ is strongly self-$F$-split if and only if $M$ is self-$F$-split and ${\rm End}_{\mathcal{A}}(C)$ is abelian.
\item $M$ is dual strongly self-$F$-split if and only if $M$ is dual self-$F$-split and ${\rm End}_{\mathcal{A}}(F)$ is abelian.
\end{enumerate}
\end{theor}

\begin{proof} (1) Assume first that $M$ is strongly self-$F$-split. Then $M$ is self-$F$-split. 
Also, by Theorem \ref{t:key}, $M\cong F\oplus C$ and $C$ is strongly self-Rickart. 
Then ${\rm End}_{\mathcal{A}}(C)$ is abelian by \cite[Proposition~2.14]{CO1}.

Conversely, assume that $M$ is self-$F$-split and ${\rm End}_{\mathcal{A}}(C)$ is abelian. 
Then $M\cong F\oplus C$ and $C$ is self-Rickart by Theorem \ref{t:key}. It follows that 
$C$ is strongly self-Rickart by \cite[Proposition~2.14]{CO1}. Hence $M$ is strongly self-$F$-split 
by Theorem \ref{t:key}. 
\end{proof}

\begin{exam} \rm Let $A=\prod_{n=1}^{\infty}\mathbb{Z}_2$, $T=\{(a_n)_{n=1}^{\infty}\in A\mid a_n \textrm{ is eventually
constant}\}$ and $I=\{(a_n)_{n=1}^{\infty}\in A\mid a_n \textrm{ is eventually
zero}\}=\bigoplus_{n=1}^{\infty}\mathbb{Z}_2$. Consider the ring 
$R=\left[\begin{smallmatrix} T&T/I \\ 0&T/I\end{smallmatrix}\right]$ and $M=R_R$. 
Then $F=\left[\begin{smallmatrix} T&T/I \\ 0&0\end{smallmatrix}\right]$ is an ideal of $R$,
hence it is a fully invariant submodule of $M$. Moreover, we have $M=F\oplus C$, where 
$C=\left[\begin{smallmatrix} 0&0 \\ 0&T/I \end{smallmatrix}\right]$. Since $C$ is projective, 
$C$ is self-Rickart by \cite[Theorem~4.7]{C-K}, and so $M$ is self-$F$-split by Theorem \ref{t:key}. 
Also, $F$ is dual self-Rickart by \cite[Example~4.1]{LRR11},
and so $M$ is dual self-$F$-split by Theorem \ref{t:key}. Since ${\rm End}_R(C)$ and 
${\rm End}_R(F)=\left[\begin{smallmatrix} T&T/I \\ 0&T/I\end{smallmatrix}\right]$ are commutative, it follows that 
$M$ is strongly self-$F$-split and dual strongly self-$F$-split by Theorem \ref{t:endab}.
\end{exam}

The following theorem generalizes \cite[Theorem~2.9]{Ungor}.

\begin{theor} Let $\A$ be an abelian category. Let $0\to F\stackrel{i}\to M\stackrel{d}\to C\to 0$
be a fully invariant short exact sequence in $\A$.
\begin{enumerate}
\item  The following are equivalent:
\begin{enumerate}[(i)]
\item $M$ is (strongly) self-$F$-split and for every pullback square
$$\SelectTips{cm}{}
\xymatrix{
P \ar[r]^j \ar[d]_f & M \ar[d]^g \\
F \ar[r]_i & M }$$
the unique morphism $l:{\rm Ker}(g)\to P$ is a (fully invariant) section.

\item $M$ is (strongly) self-Rickart and $M\cong F\oplus C$.
\end{enumerate}
\item The following are equivalent:
\begin{enumerate}[(i)]
\item $M$ is dual (strongly) $F$-split and for every pushout square
$$\SelectTips{cm}{}
\xymatrix{
M \ar[r]^d \ar[d]_g & C \ar[d]^h \\
M \ar[r]_p & Q }$$
the unique morphism $q:Q\to {\rm Coker}(g)$ is a (fully coinvariant) retraction.
\item $M$ is dual (strongly) self-Rickart and $M\cong F\oplus C$.
\end{enumerate}
\end{enumerate}
\end{theor}

\begin{proof} (1) (i) $\Rightarrow$ (ii) Assume that (i) holds. Let $g:M\to M$ be a morphism in $\A$ with kernel $k:\Ker(g)\to M$.
Consider the pullback of $g$ and $i$ as in the above diagram.
The pullback property yields a unique morphism $l:\Ker(g)\to P$ such that $k=jl$.
Since $M$ is (strongly) self-$F$-split, $j:P\to M$ is a (fully invariant) section. It follows that $k=jl$ is a section.
If $j$ and $l$ are fully invariant sections, then so is $k=jl$ by Lemma \ref{l:trans}.
Hence $M$ is (strongly) self-Rickart. Also, $M\cong F\oplus C$ by Theorem \ref{t:key}.

(ii) $\Rightarrow$ (i) Assume that (ii) holds. Consider a pullback square as above.
Denote by $k:\Ker(g)\to M$ the kernel of $g:M\to M$. The pullback property yields a unique morphism $l:\Ker(g)\to P$ such that $k=jl$.
Since $M\cong F\oplus C$ is (strongly) self-Rickart, so is $C$ by \cite[Corollary~2.11]{C-K} (\cite[Corollary~2.18]{CO1}).
Hence $M$ is (strongly) self-$F$-split by Theorem \ref{t:key}. Since $M$ is (strongly) self-Rickart,
$k=jl$ is a (fully invariant) section. It follows that $l:\Ker(g)\to P$ is a section.
If $k=jl$ is a fully invariant section, then so is $l$ by Proposition \ref{p:POPB}.
\end{proof}

The next result generalizes \cite[Proposition~2.16]{Ungor}, without the quasi-projectivity condition.

\begin{prop} Let $\A$ be an abelian category. Let $0\to F\stackrel{i}\to M\stackrel{d}\to C\to 0$
be a fully invariant short exact sequence in $\A$. Then:
\begin{enumerate}
\item $M$ is (strongly) self-$F$-split if and only if
for every subobject $K$ of $M$ with $K\subseteq F$, $M/K$ is (strongly) self-$(F/K)$-split.
\item $M$ is dual (strongly) self-$F$-split if and only if
for every subobject $K$ of $M$ with $F\subseteq K$, $K$ is dual (strongly) self-$F$-split.
\end{enumerate}
\end{prop}

\begin{proof} (1) Assume first that $M$ is (strongly) self-$F$-split. By Theorem \ref{t:key}, $M\cong F\oplus C$ and
$C$ is (strongly) self-Rickart. Let $K$ be a subobject of $M$ with $K\subseteq F$.
One may construct the following commutative diagram
$$\SelectTips{cm}{}
\xymatrix{
0 \ar[r] & F \ar[r]^-{i} \ar[d]_f & M \ar[r]^-{d} \ar[d]^g & C \ar[r] \ar@{=}[d] & 0 \\
0 \ar[r] & F/K \ar[r]_-{i'} & M/K \ar[r]_-{d'} & C \ar[r] & 0 }$$
where the rows are short exact sequences and $f,g$ are the cokernels induced by the inclusions of $K$ into $F$ and $M$.
Since $d=d'g$ is a fully coinvariant retraction, so is $d'$ by Proposition \ref{p:POPB}.
Then $i':F/K\to M/K$ is a fully invariant section, and so $M/K\cong F/K\oplus C$.
Hence $M/K$ is (strongly) self-$(F/K)$-split by Theorem \ref{t:key}.

The converse is clear.
\end{proof}

\section{Coproducts of $F$-split objects}

In general arbitrary coproducts of relative split objects are not relative split objects, 
for instance, see \cite[Example~3.1]{Ungor}. Nevertheless, 
we may give the following theorem if we impose some extra conditions on the coproducts.   

\begin{theor} \label{t:homzero} Let $\A$ be an abelian category with coproducts,
and let $(M_k)_{k\in K}$ be a family of objects of $\A$. 
Let $0\to F_k\stackrel{i_k}\to M_k\stackrel{d_k}\to C_k\to 0$ be fully invariant 
short exact sequences in $\A$ for every $k\in K$. 
\begin{enumerate} \item 
\begin{enumerate}[(i)]
\item Assume that ${\rm Hom}_{\A}(M_k, M_l)=0$ for every $k,l\in K$ with $k\neq l$. 
Then $\bigoplus_{k\in K}M_k$ is self-$(\bigoplus_{k\in K}F_k)$-split if and only if 
$M_k$ is self-$F_k$-split for every $k\in K$.
\item Then $\bigoplus_{k\in K}M_k$ is strongly self-$(\bigoplus_{k\in K}F_k)$-split if and only if 
$M_k$ is strongly self-$F_k$-split for every $k\in K$ and ${\rm Hom}_{\A}(C_k, C_l)=0$ for every $k,l\in K$ with $k\neq l$.
\end{enumerate}
\item 
\begin{enumerate}[(i)]
\item Assume that ${\rm Hom}_{\A}(M_k, M_l)=0$ for every $k,l\in K$ with $k\neq l$. 
Then $\bigoplus_{k\in K}M_k$ is dual self-$(\bigoplus_{k\in K}F_k)$-split if and only if 
$M_k$ is dual self-$F_k$-split for every $k\in K$.
\item Then $\bigoplus_{k\in K}M_k$ is dual strongly self-$(\bigoplus_{k\in K}F_k)$-split if and only if 
$M_k$ is dual strongly self-$F_k$-split for every $k\in K$ and ${\rm Hom}_{\A}(F_k, F_l)=0$ for every $k,l\in K$ with $k\neq l$.
\end{enumerate}
\end{enumerate}
\end{theor}

\begin{proof} (1) (i) The direct implication follows by Lemma \ref{l:dsum} and Corollary \ref{c:strel}.

Conversely, assume that $M_k$ is (strongly) self-$F_k$-split for every $k\in K$.
Let $g:\bigoplus_{k\in K}M_k\to \bigoplus_{k\in K}M_k$ be a morphism. 
Since ${\rm Hom}_{\A}(M_k, M_l)=0$, for every $k,l\in K$ with $k\neq l$, the matrix of $g$ has zero entries 
except for the entries $(k,k)$ with $k\in K$, which are some morphisms $g_k:M_k\to M_k$. 
Hence $g=\bigoplus_{k\in K}g_k$. Then we have pullback squares as follows:
$$\SelectTips{cm}{}
\xymatrix{
P_k \ar[r]^-{j_k} \ar[d]_{f_k} & M_k \ar[d]^{g_k} \\
F_k \ar[r]_{i_k} & M_k }
\hspace{2cm}
\xymatrix{
\bigoplus_{k\in K} P_k \ar[rr]^-{\bigoplus_{k\in K} j_k} \ar[d]_{\bigoplus_{k\in K} f_k} && 
\bigoplus_{k\in K} M_k \ar[d]^{\bigoplus_{k\in K} g_k} \\
\bigoplus_{k\in K} F_k \ar[rr]_{\bigoplus_{k\in K} i_k} && \bigoplus_{k\in K} M_k }
$$
For every $k\in K$, $M_k$ is (strongly) self-$F_k$-split, hence $j_k$ is a (fully invariant) section. 
It follows that $\bigoplus_{k\in K} j_k$ is a (fully invariant) section (by Lemma \ref{l:dsum}). 
Hence $\bigoplus_{k\in K}M_k$ is (strongly) self-$(\bigoplus_{k\in K}F_k)$-split.

(ii) By \cite[Theorem~3.6]{CO1}, $\bigoplus_{k\in K}C_k$ is strongly self-Rickart if and only if 
$C_k$ is strongly self-Rickart and ${\rm Hom}_{\A}(C_k, C_l)=0$ for every $k,l\in K$ with $k\neq l$.
Then use (i) and Theorem \ref{t:key} to derive the conclusion.
\end{proof}

\begin{exam} \rm Let us see that the zero Hom conditions from Theorem \ref{t:homzero} are not superfluous.

(a) Consider the abelian group $G=G_1\oplus G_2$, where $G_1=\mathbb{Z}_q\oplus \mathbb{Z}$ and 
$G_2=\mathbb{Z}_p$ for some distinct primes $p$ and $q$. By \cite[Lemma~1.9]{OHS}, $F_1\cong \mathbb{Z}_q$ is 
a fully invariant subgroup of $G_1$, because $\Hom_{\mathbb{Z}}(\mathbb{Z}_q,\mathbb{Z})=0$. 
Clearly, $F_2=0$ is a fully invariant subgroup of $G_2$. 
By Theorem \ref{t:key}, $G_1\cong F_1\oplus \mathbb{Z}$ is strongly self-$F_1$-split and 
$G_2=\mathbb{Z}_p$ is strongly self-$F_2$-split, because $\mathbb{Z}$ and $\mathbb{Z}_p$ are strongly self-Rickart 
\cite[Corollary~3.9]{CO1}. Note that $F_1\oplus F_2=\mathbb{Z}_q$ 
is a fully invariant subgroup of $G=G_1\oplus G_2$, because $\Hom_{\mathbb{Z}}(\mathbb{Z}_q,\mathbb{Z}\oplus \mathbb{Z}_p)=0$.
By Theorem \ref{t:key}, $G$ is not self-$F_1\oplus F_2$-split, because $\mathbb{Z}\oplus \mathbb{Z}_p$ is not 
self-Rickart \cite[Example~2.5]{LRR10}. Note that $\Hom_{\mathbb{Z}}(G_1,G_2)\neq 0$.  

(b) Consider the abelian group $G=G_1\oplus G_2$, where $G_1=\mathbb{Z}_{p^{\infty}}\oplus \mathbb{Z}_q$ and 
$G_2=\mathbb{Z}_p$ for some distinct primes $p$ and $q$. By \cite[Lemma~1.9]{OHS}, $F_1\cong \mathbb{Z}_{p^{\infty}}$ is 
a fully invariant subgroup of $G_1$, because $\Hom_{\mathbb{Z}}(\mathbb{Z}_{p^{\infty}},\mathbb{Z}_q)=0$. 
Clearly, $F_2=\mathbb{Z}_p$ is a fully invariant subgroup of $G_2$. 
By Theorem \ref{t:key}, $G_1\cong F_1\oplus \mathbb{Z}_q$ is dual strongly self-$F_1$-split and 
$G_2=\mathbb{Z}_p$ is dual strongly $0$-split, because $F_1\cong \mathbb{Z}_{p^{\infty}}$ and $\mathbb{Z}_p$ 
are dual strongly self-Rickart \cite[Corollary~3.9]{CO1}. Note that $F_1\oplus F_2=\mathbb{Z}_{p^{\infty}}\oplus \mathbb{Z}_p$ 
is a fully invariant subgroup of $G=G_1\oplus G_2$, 
because $\Hom_{\mathbb{Z}}(\mathbb{Z}_{p^{\infty}}\oplus \mathbb{Z}_p,\mathbb{Z}_q)=0$.
By Theorem \ref{t:key}, $G$ is not dual self-$F_1\oplus F_2$-split, 
because $\mathbb{Z}_{p^{\infty}}\oplus \mathbb{Z}_p$ is not 
self-Rickart \cite[Example~2.10]{LRR11}. Note that $\Hom_{\mathbb{Z}}(G_2,G_1)\neq 0$.  
\end{exam}

\begin{theor} \label{t:ds} Let $\mathcal{A}$ be an abelian category. 
Let $M$ and $N$ be objects of $\A$, $N=\bigoplus_{k=1}^n N_k$ a direct sum decomposition, 
and $0\to F\to N\to C\to 0$ a fully invariant short exact sequence in $\A$. Then:
\begin{enumerate} 
\item $N$ is (strongly) $M$-$F$-split if and only if 
$N_k$ is (strongly) $M$-$(F\cap N_k)$-split for every $k\in \{1,\dots,n\}$.
\item $N$ is dual (strongly) $M$-$F$-split if and only if 
$N_k$ is dual (strongly) $M$-$((F+N_k)/N_k)$-split for every $k\in \{1,\dots,n\}$.
\end{enumerate}
\end{theor}

\begin{proof} (1) The direct implication follows by Corollary~\ref{c:strel}.

For the sake of clarity, we prove the converse for $n=2$, the general case following inductively. 
Assume that $N_k$ is (strongly) $M$-$(F\cap N_k)$-split for $k=1,2$. By Proposition \ref{p:fids} 
we have $F\cong (F\cap N_1)\oplus (F\cap N_2)$. Denote by 
$i=\left[\begin{smallmatrix} i_1&0 \\ 0&i_2 \end{smallmatrix}\right]:(F\cap N_1)\oplus (F\cap N_2)\to N_1\oplus N_2$
the inclusion morphism. 
Let $g=\left[\begin{smallmatrix} g_1\\g_2 \end{smallmatrix}\right]:M\to N=N_1\oplus N_2$ be a morphism in $\A$.
Consider the following pullback squares:
$$\SelectTips{cm}{}
\xymatrix{
P_1 \ar[r]^-{j_1} \ar[d]_{f_1} & M \ar[d]^{g_1} \\
F\cap N_1 \ar[r]_-{i_1} & N_1 }
\hspace{2cm}
\xymatrix{
P_2 \ar[r]^-{j_2} \ar[d]_{f_2} & M \ar[d]^{g_2} \\
F\cap N_2 \ar[r]_-{i_2} & N_2 }$$
Since $N_1$ is (strongly) $M$-$(F\cap N_1)$-split and $N_2$ is (strongly) $M$-$(F\cap N_2)$-split, 
$j_1$ and $j_2$ are (fully invariant) sections. Then there exists an epimorphism $p_2:M\to P_2$ such that $p_2j_2=1_{P_2}$.
Consider the pullback of $g_1j_2:P_2\to N_1$ and $i_1$ in order to get the outer part of the following commutative diagram:
$$\SelectTips{cm}{}
\xymatrix{
P \ar[r]^u \ar@{-->}[d]_w \ar@/_2pc/[dd]_v & P_2 \ar[d]^{j_2} \\
P_1 \ar[r]^-{j_1} \ar[d]_{f_1} & M \ar[d]^{g_1} \\
F\cap N_1 \ar[r]_-{i_1} & N_1 }
$$
By Theorem \ref{t:epimono}, $N_1$ is $P_2$-$F_1$-split, hence $u$ is a (fully invariant) section.
Since the lower square is a pullback, so is the upper square by \cite[Lemma~5.1]{Kelly}. 
Also, there exists a unique morphism $w:P\to P_1$ such that $f_1w=v$ and $j_1w=j_2u$.

We claim that the following square is a pullback:
$$\SelectTips{cm}{}
\xymatrix{
P \ar[r]^-{j_2u} \ar[d]_{\left[\begin{smallmatrix} v\\f_2u \end{smallmatrix}\right]} & 
M \ar[d]^{\left[\begin{smallmatrix} g_1\\g_2 \end{smallmatrix}\right]} \\
(F\cap N_1)\oplus (F\cap N_2) \ar[r]_-{\left[\begin{smallmatrix} i_1&0 \\ 0&i_2 \end{smallmatrix}\right]} & N_1\oplus N_2 }
$$
It is commutative, because we have \[\left[\begin{smallmatrix} i_1&0 \\ 0&i_2 \end{smallmatrix}\right]
\left[\begin{smallmatrix} v\\f_2u \end{smallmatrix}\right]=
\left[\begin{smallmatrix} i_1v\\i_2f_2u \end{smallmatrix}\right]=
\left[\begin{smallmatrix} g_1j_2u\\g_2j_2u \end{smallmatrix}\right]=
\left[\begin{smallmatrix} g_1\\g_2 \end{smallmatrix}\right]j_2u.\] 
Now let $\left[\begin{smallmatrix} \alpha_1\\ \alpha_2 \end{smallmatrix}\right]:X\to (F\cap N_1)\oplus (F\cap N_2)$ 
and $\beta:X\to M$ be morphisms such that 
$\left[\begin{smallmatrix} i_1&0 \\ 0&i_2 \end{smallmatrix}\right]
\left[\begin{smallmatrix} \alpha_1\\ \alpha_2 \end{smallmatrix}\right]=
\left[\begin{smallmatrix} g_1\\g_2 \end{smallmatrix}\right]\beta$. 
Hence $i_1\alpha_1=g_1\beta$ and $i_2\alpha_2=g_2\beta$.
The pullback properties of the first two squares from the proof yield unique morphisms
$\gamma_1:X\to P_1$ such that $f_1\gamma_1=\alpha_1$ and $j_1\gamma_1=\beta$,
and $\gamma_2:X\to P_2$ such that $f_2\gamma_2=\alpha_2$ and $j_2\gamma_2=\beta$.
Hence $j_1\gamma_1=j_2\gamma_2$. Since the square $PP_2P_1M$ is a pullback, 
there exists a unique morphism $\gamma:X\to P$ such that $w\gamma=\gamma_1$ and $u\gamma=\gamma_2$. 
It follows that
$\left[\begin{smallmatrix} v\\f_2u \end{smallmatrix}\right]\gamma=
\left[\begin{smallmatrix} f_1w\gamma \\ f_2\gamma_2 \end{smallmatrix}\right]=
\left[\begin{smallmatrix} f_1\gamma_1 \\ f_2\gamma_2 \end{smallmatrix}\right]=
\left[\begin{smallmatrix} \alpha_1\\ \alpha_2 \end{smallmatrix}\right]$
and $j_2u\gamma=j_2\gamma_2=\beta$. For uniqueness, if there exists a morphism $\gamma':X\to P$ 
such that $\left[\begin{smallmatrix} v\\f_2u \end{smallmatrix}\right]\gamma'=
\left[\begin{smallmatrix} \alpha_1\\ \alpha_2 \end{smallmatrix}\right]$
and $j_2u\gamma'=\beta$, then we have $j_2u\gamma=j_2u\gamma'$. 
Then $\gamma=\gamma'$, because $j_2$ and $u$ are monomorphisms.
Thus, the required square is a pullback. 

Finally, since $j_2$ and $u$ are (fully invariant) sections, so is $j_2u$ (by Lemma \ref{l:trans}).
This shows that $N$ is (strongly) $M$-$F$-split.
\end{proof}

\begin{cor} \label{c:fg} Let $\mathcal{A}$ be an abelian category. 
Let $M$ and $N$ be objects of $\A$, $N=\bigoplus_{k\in K} N_k$ a direct sum decomposition, 
and $0\to F\to N\to C\to 0$ a fully invariant short exact sequence in $\A$. Then:
\begin{enumerate} \item Assume that $M$ is finitely generated. Then $N$ is (strongly) $M$-$F$-split if and only if 
$N_k$ is (strongly) $M$-$(F\cap N_k)$-split for every $k\in K$.
\item Assume that $N$ is finitely cogenerated. Then $N$ is dual (strongly) $M$-$F$-split if and only if 
$N_k$ is dual (strongly) $M$-$((F+N_k)/N_k)$-split for every $k\in K$.
\end{enumerate}
\end{cor}

\begin{proof} (1) The direct implication follows by Corollary~\ref{c:strel}. 

Conversely, assume that $N_k$ is (strongly) $M$-$(F\cap N_k)$-split for every $k\in K$.
Let $g:M\to N=\bigoplus_{k\in K} N_k$ be a morphism in $\mathcal{A}$. 
Since $M$ is finitely generated, we may write 
$g=lg'$ for some morphism $g':M\to \bigoplus_{k\in F} N_k$ and inclusion morphism $l:\bigoplus_{k\in A} N_k\to N$, 
where $A$ is a finite subset of $K$. Consider the following commutative diagram:
$$\SelectTips{cm}{}
\xymatrix{
P \ar[r]^j \ar[d]_{f'} & M \ar[d]^{g'} \\
F\cap \left( \bigoplus_{k\in A} N_k\right ) \ar[r]^-{i'} \ar[d]_{l'} & \bigoplus_{k\in A} N_k \ar[d]^{l} \\
F \ar[r]_-{i} & N }
$$
in which the lower and the upper rectangles are pullbacks. Then the outer rectangle is a pullback \cite[Lemma~5.1]{Kelly}.
By Corollary \ref{c:Fds}, $i'$ is a fully invariant kernel. 
By Theorem~\ref{t:ds}, $\bigoplus_{k\in A} N_k$ is (strongly) $M$-$(F\cap \left( \bigoplus_{k\in A} N_k\right ))$-split,
hence $j$ is a (fully invariant) section. It follows that $N$ is (strongly) $M$-$F$-split.
\end{proof}

\begin{remark} \rm Note that in general Theorem \ref{t:ds} does not hold for arbitrary coproducts. 
For an example in the case $F=0$ see \cite[Example~3.3]{C-K}.
\end{remark}

Let $\A$ be an abelian category and let $r$ be a preradical of $\A$. 
Recall that an object $A$ of $\A$ is called \emph{$r$-torsion} if $r(A)=A$, 
and \emph{$r$-torsionfree} if $r(A)=0$. The class of $r$-torsion objects is closed under factor objects and coproducts,
while the class of $r$-torsionfree objects is closed under subobjects and products. 
The preradical $r$ is called \emph{superhereditary} if $r$ is hereditary and 
the class of $r$-torsion objects of $\A$ is closed under products (e.g., see \cite{BKN}). 

\begin{theor} \label{t:dsprerad} Let $\A$ be an abelian category and let $r$ be a preradical of $\A$. 
\begin{enumerate}
\item Let $M$ be an object of $\A$ and let $(N_k)_{k\in K}$ be a family of objects of $\A$ having a coproduct. 
\begin{enumerate}[(i)] 
\item Assume that $M$ has SSIP for (fully invariant) direct summands containing $r(M)$. Then 
$\bigoplus_{k\in K} N_k$ is (strongly) $M$-$r(\bigoplus_{k\in K} N_k)$-split if and only if 
$N_k$ is (strongly) $M$-$r(N_k)$-split for every $k\in K$.
\item Assume that $K$ is finite and $M$ has SIP for (fully invariant) direct summands containing $r(M)$. 
Then $\bigoplus_{k\in K} N_k$ is (strongly) $M$-$r(\bigoplus_{k\in K} N_k)$-split if and only if 
$N_k$ is (strongly) $M$-$r(N_k)$-split for every $k\in K$.
\end{enumerate}
\item Let $N$ be an object of $\A$ and let $(M_k)_{k\in K}$ be a family of objects of $\A$ having a product.
\begin{enumerate}[(i)]
\item Assume that $r$ is superhereditary and $N$ has SSSP for (fully invariant) direct summands contained in $r(N)$.  
Then $N$ is dual (strongly) $\prod_{k\in K} M_k$-$r(N)$-split if and only if 
$N$ is dual (strongly) $M_k$-$r(N)$-split for every $k\in K$.
\item Assume that $K$ is finite and $N$ has SSP for (fully invariant) direct summands contained in $r(N)$.  
Then $N$ is dual (strongly) $\prod_{k\in K} M_k$-$r(N)$-split if and only if 
$N$ is dual (strongly) $M_k$-$r(N)$-split for every $k\in K$.
\end{enumerate}
\end{enumerate}
\end{theor}

\begin{proof} (1) (i) The direct implication follows by Corollary~\ref{c:strel}. 

Conversely, assume that $N_k$ is (strongly) $M$-$r(N_k)$-split for every $k\in K$. 
Denote by $i:r(\bigoplus_{k\in K} N_k)\to \bigoplus_{k\in K} N_k$ the inclusion monomorphism 
and by $d:\bigoplus_{k\in K} N_k\to C$ its cokernel. 
Let $g:M\to \bigoplus_{k\in K} N_k$ be a morphism in $\A$. Consider the pullback of $g$ and $i$ 
in order to get morphisms $f:P\to r(\bigoplus_{k\in K} N_k)$ and $j:P\to M$. For every $k\in K$, 
denote by $p_k:\bigoplus_{k\in K} N_k\to N_k$ the canonical projection and $g_k=p_kg:M\to N_k$. 
For every $k\in K$, denote by $i_k:r(N_k)\to N_k$ the inclusion monomorphism and by 
$d_k:N_k\to C_k$ its cokernel. For every $k\in K$, consider the following pullback diagram:
$$\SelectTips{cm}{}
\xymatrix{
P_k \ar[r]^{j_k} \ar[d]_{f_k} & M \ar[d]^{g_k} \\
r(N_k) \ar[r]_-{i_k} & N_k }
$$
Let $u:r(M)\to M$ be the inclusion monomorphism. 
Since $r$ is a preradical, each morphism $g_k:M\to N_k$ induces a morphism $l_k:r(M)\to r(N_k)$
such that $g_ku=i_kl_k$. By the pullback property of the upper square it follows that 
$r(M)\subseteq P_k$ for every $k\in K$.
Since each $N_k$ is (strongly) $M$-$r(N_k)$-split, each $j_k:P_k\to M$ is a (fully invariant) section. 
Note that $r(\bigoplus_{k\in K} N_k)=\bigoplus_{k\in K} r(N_k)$ \cite[I.1.2]{BKN}. 
It follows that \[{\rm Ker}(dg)=P=\bigcap_{k\in K} P_k=\bigcap_{k\in K} {\rm Ker}(d_kg_k)\] 
and $M$ has SSIP for (fully invariant) direct summands containing $r(M)$, 
it follows that $j={\rm ker}(dg):P\to M$ is a (fully invariant) section. 
Hence $\bigoplus_{k\in K} N_k$ is (strongly) $M$-$r(\bigoplus_{k\in K} N_k)$-split by Lemma \ref{l:kc}.

(ii) This is similar to (i).

(2) Let us only note that if $r$ is superhereditary, then $r(\prod_{k\in K} M_k)=\prod_{k\in K}{r(M_k)}$ 
by definition and \cite[I.1.2]{BKN}. Then the proof is dual to (1). 
\end{proof}

\begin{cor} Let $\A$ be an abelian category, let $r$ be a preradical of $\A$, let $M_1,\dots,M_k$ be objects of $\A$
and $l\in \{1,\dots,n\}$.
\begin{enumerate}
\item $\bigoplus_{k=1}^n M_k$ is (strongly) $M_l$-$r(\bigoplus_{k=1}^n M_k)$-split if and only if 
$M_k$ is (strongly) $M_l$-$r(M_k)$-split for every $k\in \{1,\dots,n\}$.
\item $M_l$ is dual (strongly) $\bigoplus_{k=1}^n M_k$-$r(M_l)$-split if and only if 
$M_l$ is dual (strongly) $M_k$-$r(M_l)$-split for every $k\in \{1,\dots,n\}$.
\end{enumerate}
\end{cor}

\begin{proof} This follows by Corollary \ref{c:SIP} and Theorem \ref{t:dsprerad}. 
\end{proof}

\section{Classes all of whose objects are $F$-split}

Recall that an abelian category $\mathcal{A}$ is called \emph{spectral} if every short exact sequence in $\mathcal{A}$ splits. 

\begin{theor} \label{t:spectralR} Let $\mathcal{A}$ be an abelian category.
\begin{enumerate}
\item The following are equivalent:
\begin{enumerate}[(i)] 
\item $\mathcal{A}$ is spectral.
\item $\mathcal{A}$ has enough injectives and $N$ is $M$-$F$-split for every objects $M$ and $N$ of $\mathcal{A}$ 
and for every fully invariant subobject $F$ of $N$.
\item $\mathcal{A}$ has enough injectives and every object $N$ of $\mathcal{A}$ is self-$F$-split 
for every fully invariant subobject $F$ of $N$.
\item $\mathcal{A}$ has enough injectives and $N$ is $M$-$F$-split for every objects $M$ and $N$ of $\mathcal{A}$ with
$N$ injective and for every fully invariant subobject $F$ of $N$.
\item $\mathcal{A}$ has enough injectives and every injective object $N$ of $\mathcal{A}$ is self-$F$-split 
for every fully invariant subobject $F$ of $N$.
\end{enumerate}
\item The following are equivalent:
\begin{enumerate}[(i)] 
\item $\mathcal{A}$ is spectral.
\item $\mathcal{A}$ has enough projectives and $N$ is dual $M$-$F$-split for every objects $M$ and $N$ of $\mathcal{A}$
and for every fully invariant subobject $F$ of $N$.
\item $\mathcal{A}$ has enough projectives and every object $N$ of $\mathcal{A}$ is dual self-$F$-split
for every fully invariant subobject $F$ of $N$.
\item $\mathcal{A}$ has enough projectives and $N$ is dual $M$-Rickart for every objects $M$ and $N$ of $\mathcal{A}$
with $M$ projective and for every fully invariant subobject $F$ of $N$.
\item $\mathcal{A}$ has enough projectives and every projective object $N$ of $\mathcal{A}$ is dual self-$F$-split 
for every fully invariant subobject $F$ of $N$.
\end{enumerate}
\end{enumerate}
\end{theor}

\begin{proof} (1) (i)$\Rightarrow$(ii)$\Rightarrow$(iii)$\Rightarrow$(v) and (ii)$\Rightarrow$(iv)$\Rightarrow$(v)
are clear. 

(v)$\Rightarrow$(i) Assume that $\mathcal{A}$ has enough injectives and every injective object $N$ of $\mathcal{A}$ is self-$F$-split 
for every fully invariant subobject $F$ of $N$. Taking $F=0$, it follows that 
every injective object $N$ of $\mathcal{A}$ is self-Rickart, hence $\mathcal{A}$ is spectral by \cite[Theorem~4.1]{C-K}.
\end{proof}

The category ${\rm Mod}(R)$ is a locally finitely generated (i.e., it has a family of finitely generated generators) 
Grothendieck category with enough injectives and enough projectives. 
It is spectral if and only if $R$ is semisimple \cite[Chapter~V, Proposition~6.7]{S}. 

\begin{cor} \label{c:semis} Let $R$ be a unitary ring. Then the following are equivalent:
\begin{enumerate}[(i)]
\item $R$ is semisimple.
\item Every right $R$-module $N$ is (strongly) self-$F$-split for every fully invariant submodule $F$ of $N$.
\item Every injective right $R$-module $N$ is (strongly) self-$F$-split for every fully invariant submodule $F$ of $N$.
\item Every right $R$-module $N$ is dual (strongly) self-$F$-split for every fully invariant submodule $F$ of $N$.
\item Every projective right $R$-module $N$ is dual (strongly) self-$F$-split for every fully invariant submodule $F$ of $N$.
\end{enumerate}
\end{cor}

\begin{proof} We only discuss the strong versions of the equivalences (i)$\Leftrightarrow$(ii)$\Leftrightarrow$(iii), 
the remaining part following by Theorem \ref{t:spectralR} and by duality.

(i)$\Rightarrow$(ii) Assume that $R$ is semisimple. Then every right $R$-module is semisimple, 
hence for every right $R$-module $N$, every submodule of $N$ is fully invariant. 
Thus every right $R$-module is (weak) duo \cite{OHS}. By Theorem \ref{t:spectralR}, 
every right $R$-module is self-$0$-split, i.e., self-Rickart. Then 
every right $R$-module is strongly self-Rickart \cite[Corollary~2.10]{CO1}. Now using Theorem \ref{t:key}, 
it follows that every right $R$-module $N$ is strongly self-$F$-split for every fully invariant submodule $F$ of $N$.

(ii)$\Rightarrow$(iii) This is clear.

(iii)$\Rightarrow$(i) This follows by Theorem \ref{t:spectralR}.
\end{proof}

Let ${}^C\mathcal{M}$ be the category of left comodules over a coalgebra $C$ over a field \cite{DNR}. 
Then ${}^C\mathcal{M}$ is a locally finite (i.e., it has a family of generators of finite length) Grothendieck category 
with enough injectives. It has enough projectives if and only if $C$ is left semiperfect
\cite[Theorem~3.2.3]{DNR}. It is spectral if and only if ${}^C\mathcal{M}$ is semisimple if and
only if $C$ is cosemisimple. Now one can deduce the following corollary similarly to Corollary \ref{c:semis}.

\begin{cor} Let $C$ be a coalgebra over a field. Then the following are equivalent:
\begin{enumerate}[(i)] 
\item $C$ is cosemisimple.
\item Every left $C$-comodule $N$ is (strongly) self-$F$-split for every fully invariant subcomodule $F$ of $N$.
\item Every injective left $C$-comodule $N$ is (strongly) self-$F$-split for every fully invariant subcomodule $F$ of $N$.
\item $C$ is left semiperfect and every left $C$-comodule $N$ is dual (strongly) self-$F$-split 
for every fully invariant subcomodule $F$ of $N$.
\item $C$ is left semiperfect and every projective left $C$-comodule $N$ is dual (strongly) self-$F$-split 
for every fully invariant subcomodule $F$ of $N$.
\end{enumerate}
\end{cor}

A Grothendieck category $\mathcal{A}$ is called a \emph{$V$-category} if every simple object of $\A$ is injective
\cite{DNV}, and a \emph{regular} category if every object $B$ of $\mathcal{A}$ is regular in the sense that every
short exact sequence of the form $0\to A\to B\to C\to 0$ is pure in $\mathcal{A}$ (i.e., every finitely presented object
is projective with respect to it) \cite[p.~313]{Wis}. 

\begin{theor} \label{t:VR} Let $\mathcal{A}$ be a locally finitely generated Grothendieck category. Then:
\begin{enumerate}
\item $\mathcal{A}$ is a $V$-category if and only if every finitely cogenerated (injective) object $N$ of $\mathcal{A}$ 
is self-$F$-split for every fully invariant subobject $F$ of $N$.
\item $\mathcal{A}$ is a regular category if and only if every finitely generated (projective) object $N$ of $\mathcal{A}$ 
is dual self-$F$-split for every finitely generated fully invariant subobject $F$ of $N$.
\end{enumerate}
\end{theor} 

\begin{proof} (1) If $\mathcal{A}$ is a $V$-category, then $\mathcal{A}$ has a semisimple cogenerator
\cite[Theorem~2.3]{DNV}. Hence every finitely cogenerated (injective) object $N$ of $\mathcal{A}$ is semisimple, 
and so $N$ is self-$F$-split for every fully invariant subobject $F$ of $N$. 

The remaining part of the proof follows by \cite[Theorem~4.4]{C-K}, using Remark \ref{r:Rickart}.

(2) We give its proof, since it is not completely dual to (1). 

Assume that $\mathcal{A}$ is a regular category. Let $N$ be a finitely generated object of $\A$, 
$F$ a finitely generated fully invariant subobject of $N$ and $g:N\to N$ a morphism in $\mathcal{A}$. 
Then ${\rm Im}(gi)$ is a finitely generated subobject of the regular object $N$, 
hence ${\rm Im}(gi)$ is a direct summand of $N$ by \cite[37.4]{Wis}, 
whose proof works in locally finitely generated Grothendieck categories. 
Hence ${\rm coker}(gi)$ is a retraction, and so $N$ is dual self-$F$-split by Lemma \ref{l:kc}.

The remaining part of the proof follows by \cite[Theorem~4.4]{C-K}, using Remark \ref{r:Rickart}.
\end{proof}

\begin{cor} \label{c:VRmod} Let $R$ be a unitary ring. Then:
\begin{enumerate} \item $R$ is a right $V$-ring if and only if every finitely cogenerated (injective) right $R$-module $N$ 
is self-$F$-split for every fully invariant submodule $F$ of $N$.
\item $R$ is a von Neumann regular ring if and only if every finitely generated (projective) right $R$-module $N$ 
is dual self-$F$-split for every finitely generated fully invariant submodule $F$ of $N$.
\end{enumerate}
\end{cor} 

\begin{proof} Note that ${\rm Mod}(R)$ is a $V$-category if and only if $R$ is a right $V$-ring, and a regular
category if and only if $R$ is a von Neumann regular ring. Then use Theorem~\ref{t:VR}.
\end{proof}

\begin{remark} \rm If $R$ is a commutative unitary ring, then $V$-rings coincide with von Neumann regular rings
\cite[23.5]{Wis}, hence all statements from Corollary~\ref{c:VRmod} are equivalent. 
\end{remark}

\begin{cor} \label{c:VRcomod} Let $C$ be a coalgebra over a field. Then the following are equivalent:
\begin{enumerate}[(i)] 
\item $C$ is cosemisimple.
\item Every finitely cogenerated (injective) left $C$-comodule $N$ is self-$F$-split 
for every fully invariant subcomodule $F$ of $N$.
\item Every finitely generated (projective) left $C$-comodule is dual self-$F$-split
for every finitely generated fully invariant subcomodule $F$ of $N$.
\end{enumerate}
\end{cor} 

\begin{proof} Note that the comodule category ${}^C\mathcal{M}$ is a $V$-category if and only if 
${}^C\mathcal{M}$ is a regular category if and only if $C$ is cosemisimple by \cite[Proposition~2.3]{Wang}
and the proof of \cite[Corollary~4.6]{C-K}. Then use Theorem~\ref{t:VR}.
\end{proof}

Recall that an abelian category $\mathcal{A}$ is called \emph{(semi)hereditary} 
if every (finitely generated) subobject of a projective object is projective, 
and \emph{co(semi)hereditary} if every (finitely cogenerated) factor object of an injective object is injective. 

\begin{theor} \label{t:heredR} Let $\mathcal{A}$ be an abelian category.
\begin{enumerate}
\item Assume that $\mathcal{A}$ has enough projectives. Then $\mathcal{A}$ is (semi)hereditary if and only if 
every (finitely generated) projective object $N$ of $\mathcal{A}$ is self-$F$-split 
for every fully invariant subobject $F$ of $N$.
\item Assume that $\mathcal{A}$ has enough injectives. Then $\mathcal{A}$ is co(semi)hereditary if and only if
every (finitely cogenerated) injective object $N$ of $\mathcal{A}$ is dual self-$F$-split
for every fully invariant subobject $F$ of $N$.
\end{enumerate}
\end{theor}

\begin{proof} (1) Assume that $\mathcal{A}$ is (semi)hereditary. 
Let $N$ be a (finitely generated) projective object of $\A$, 
and $0\to F\stackrel{i}\to N\stackrel{d}\to C\to 0$ a fully invariant short exact sequence in $\A$. 
Let $g:N\to N$ be a morphism in $\mathcal{A}$. If $N$ is finitely generated, then so is ${\rm Im}(dg)$. 
By Theorem \ref{t:key} we have $N\cong F\oplus C$, hence $C$ is (finitely generated) projective. 
Since $\mathcal{A}$ is (semi)hereditary, ${\rm Im}(dg)\subseteq C$ is projective. 
Then ${\rm Ker}(dg)$ is a direct summand of $N$, and so $N$ is self-$F$-split by Lemma \ref{l:kc}. 

The remaining part of the proof follows by \cite[Theorem~4.7]{C-K}, using Remark \ref{r:Rickart}.
\end{proof}

\begin{cor} \label{c:her} Let $R$ be a unitary ring. Then the following are equivalent:
\begin{enumerate}[(i)]
\item $R$ is right hereditary.
\item Every projective right $R$-module $N$ is self-$F$-split for every fully invariant submodule $F$ of $N$.
\item Every injective right $R$-module $N$ is dual self-$F$-split for every fully invariant submodule $F$ of $N$.
\end{enumerate}
\end{cor}

\begin{proof} Note that ${\rm Mod}(R)$ is (co)hereditary if and only if the ring $R$ is right hereditary. 
Then use Theorem \ref{t:heredR}.
\end{proof}

\begin{cor} Let $C$ be a coalgebra over a field. Then:
\begin{enumerate}
\item If $C$ is left semiperfect, then $C$ is hereditary if and only if every projective left $C$-comodule $N$
is self-$F$-split for every fully invariant subcomodule $F$ of $N$.
\item $C$ is hereditary if and only if every injective left $C$-comodule $N$ is dual self-$F$-split
for every fully invariant subcomodule $F$ of $N$.
\end{enumerate}
\end{cor}

\begin{proof} Note that the comodule category ${}^C\mathcal{M}$ is (co)hereditary if and only if 
$C$ is a (left and right) hereditary coalgebra \cite{NTZ}. Then use Theorem \ref{t:heredR}.
\end{proof}

Next we give some results in the case of abelian categories with a (projective) generator or an (injective) cogenerator. 
Recall that an object $M$ of an abelian category is called \emph{(semi)hereditary} if 
every (finitely generated) subobject of $M$ is projective, and 
\emph{co(semi)hereditary} if every (finitely cogenerated) factor object of $M$ is injective.

\begin{theor} \label{t:her} Let $\A$ be an abelian category, and let $r$ be a preradical of $\A$.
\begin{enumerate}
\item Assume that $\A$ has a generator $G$ and enough injectives, and $r$ is a radical. Then the following are equivalent:
\begin{enumerate}[(i)]
\item $G/r(G)$ is a (semi)hereditary object of $\A$.
\item $G=r(G)\oplus B$ for some (semi)hereditary object $B$ of $\A$.
\item For every (finite) direct sum $M=G^{(I)}$, $M=r(M)\oplus D$ for some (semi) hereditary object $D$ of $\A$.
\item For every (finitely generated) projective object $M$, $M=r(M)\oplus D$ for some (semi)hereditary object $D$ of $\A$.
\item $r(G)$ is a direct summand of $G$ and every (finitely generated) $r$-torsionfree projective object of $\A$ is (semi)hereditary.
\item $r(G)$ is a direct summand of $G$ and every (finitely generated) $r$-torsionfree projective object of $\A$ is self-Rickart.
\item Every (finite) direct sum $M=G^{(I)}$ is self-$r(M)$-split.
\item Every (finitely generated) projective object $M$ is self-$r(M)$-split.
\end{enumerate}
\item Assume that $\A$ has a cogenerator $G$ and enough projectives. Then the following are equivalent:
\begin{enumerate}[(i)]
\item $r(G)$ is a cosemihereditary object of $\A$.
\item $r(G)$ is a cosemihereditary direct summand of $G$.
\item For every finite direct sum $M=G^n$, $r(M)$ is a cosemihereditary direct summand of $M$.
\item For every finitely cogenerated injective object $M$, $r(M)$ is a cosemihereditary direct summand of $M$.
\item $r(G)$ is a direct summand of $G$ and every finitely cogenerated $r$-torsion injective object of $\A$ is cosemihereditary.
\item $r(G)$ is a direct summand of $G$ and every finitely cogenerated $r$-torsion injective object of $\A$ is dual self-Rickart.
\item Every finite direct sum $M=G^n$ is dual self-$r(M)$-split.
\item Every finitely cogenerated injective object $M$ is dual self-$r(M)$-split.
\end{enumerate}
\end{enumerate}
\end{theor}

\begin{proof} (1) (i)$\Rightarrow$(ii) Assume that $G/r(G)$ is a (semi)hereditary object of $\A$. 
It follows that $G/r(G)$ is projective, and so $G=r(G)\oplus B$ with $B\cong G/r(G)$ (semi)hereditary.

(ii)$\Rightarrow$(iii) Assume that $G=r(G)\oplus B$ for some (semi)hereditary object $B$ of $\A$. 
Let $M=G^{(I)}$ for some (finite) set $I$. Then $M=r(G)^{(I)}\oplus B^{(I)}=r(M)\oplus B^{(I)}$. 
Since $B$ is (semi)hereditary, so is $B^{(I)}$ by \cite[39.3, 39.7]{Wis}, whose proofs work in abelian categories 
with enough injectives. 

(iii)$\Rightarrow$(iv) Assume that (iii) holds. Let $M$ be a (finitely generated) projective object of $\A$. 
Then $M$ is a direct summand of a (finite) direct sum $F=G^{(I)}$, say $F=M\oplus N$. 
It follows that $r(F)=r(M)\oplus r(N)$ and $F/r(F)\cong M/r(M)\oplus N/r(N)$. 
Since $F/r(F)$ is (semi)hereditary, so is $M/r(M)$. Since $r(F)$ is a direct summand of $F$, 
$r(M)$ must be a direct summand of $M$. 

(iv)$\Rightarrow$(v) This is clear.

(v)$\Rightarrow$(vi) Assume that (v) holds. Let $M$ be a (finitely generated) $r$-torsionfree projective object of $\A$.
Let $f:M\to M$ be a morphism in $\A$. Since $M$ is (semi)hereditary, $M/{\rm Ker}(f)\cong {\rm Im}(f)$ is projective,
and so ${\rm Ker}(f)$ is a direct summand of $M$. Hence $M$ is self-Rickart. 

(vi)$\Rightarrow$(vii) Assume that (vi) holds. Let $M=G^{(I)}$ be a (finite) direct sum. 
Since $G=r(G)\oplus B$, it follows that $M=r(M)\oplus B^{(I)}$. Since 
the (finitely generated) $r$-torsionfree projective object $B^{(I)}$ is self-Rickart, 
$M$ is self-$r(M)$-split by Theorem \ref{t:key}.

(vii)$\Rightarrow$(viii) This follows by Corollary \ref{c:strel}.

(viii)$\Rightarrow$(i) Assume that (viii) holds. By Theorem \ref{t:key}, $G=r(G)\oplus B$ 
for some self-Rickart object $B$ of $\A$. We show that $B$ is (semi)hereditary. To this end,
let $A$ be a (finitely generated) subobject of $B$. 
Then $A\cong M/K$ for some (finite) direct sum $M=G^{(I)}$ and subobject $K$ of $M$. 
Denote by $\varphi:A\to M/K$ an isomorphism, by $p:M\to M/K$ the natural epimorphism, 
by $k:A\to M$ and $i:r(M)\to M$ the inclusion monomorphisms and $g=k\varphi p:M\to M$. 
Consider the pullback of $i$ and $g$ in order to get the following commutative square:
$$\SelectTips{cm}{}
\xymatrix{
P \ar[r]^-{j} \ar[d]_{f} & M \ar[d]^{g} \\
r(M) \ar[r]_-{i} & M }$$
Since $r$ is a radical, $B\cong G/r(G)$ is an $r$-torsionfree object, hence so is $A$. 
This implies that $if=gj=0$. Then $pj=0$, and so there exists a unique morphism $\alpha:P\to K$ such that $k\alpha=j$, 
which implies that $\alpha$ is a monomorphism. 
The pullback property yields a unique morphism $\beta:K\to P$ such that $f\beta=0$ and $j\beta=k$. 
Then $k\alpha\beta=k$, hence $\alpha\beta=1_K$, which shows that $P\cong K$. 
Since $M$ is projective, it is self-$r(M)$-split. Then $j:P\to M$ is a section, and so 
$A$ is projective. Hence $G/r(G)\cong B$ is (semi)hereditary.

(2) In general only the semihereditary case from (1) may be dualized, because 
$r$ may not preserve products (unless it is superhereditary), and an arbitrary product of 
cohereditary objects may not be cohereditary (see \cite[18.2]{Wis} and the proofs of \cite[39.3, 39.6, 39.7]{Wis}).
\end{proof}

\begin{cor} Let $r$ be a radical of ${\rm Mod}(R)$. Then the following are equivalent:
\begin{enumerate}[(i)]
\item $R/r(R)$ is a (semi)hereditary right $R$-module.
\item $R=r(R)\oplus B$ for some (semi)hereditary right ideal $B$ of $R$.
\item For every (finitely generated) free right $R$-module $M$, $M=r(M)\oplus D$ for some (semi) hereditary right $R$-module $D$.
\item For every (finitely generated) projective right $R$-module $M$, $M=r(M)\oplus D$ for some (semi)hereditary right $R$-module $D$.
\item $r(R)$ is a direct summand of $R$ and every (finitely generated) $r$-torsionfree projective right $R$-module is (semi)hereditary.
\item $r(R)$ is a direct summand of $R$ and every (finitely generated) $r$-torsionfree projective right $R$-module is self-Rickart.
\item Every (finitely generated) free right $R$-module $M$ is self-$r(M)$-split.
\item Every (finitely generated) projective right $R$-module $M$ is self-$r(M)$-split.
\end{enumerate}
\end{cor}

\begin{proof} This follows by Theorem \ref{t:her}, noting that $R$ is a generator of ${\rm Mod}(R)$.
\end{proof}

\begin{cor} Let $C$ be left semiperfect and let $r$ be a preradical of ${}^C\mathcal{M}$. Then the following are equivalent:
\begin{enumerate}[(i)]
\item $r(C)$ is a cosemihereditary left $C$-comodule.
\item $r(C)$ is a cosemihereditary direct summand of $C$.
\item For every finitely cogenerated free left $C$-comodule $M$, $r(M)$ is a cosemihereditary direct summand of $M$.
\item For every finitely cogenerated injective left $C$-comodule $M$, $r(M)$ is a cosemihereditary direct summand of $M$.
\item $r(C)$ is a direct summand of $C$ and every finitely cogenerated $r$-torsion injective left $C$-comodule is cosemihereditary.
\item $r(C)$ is a direct summand of $C$ and every finitely cogenerated $r$-torsion injective left $C$-comodule is dual self-Rickart.
\item Every finitely cogenerated free left $C$-comodule $M$ is dual self-$r(M)$-split.
\item Every finitely cogenerated injective left $C$-comodule $M$ is dual self-$r(M)$-split.
\end{enumerate}
\end{cor}

\begin{proof} This follows by Theorem \ref{t:her}, noting that $C$ is a cogenerator of ${}^C\mathcal{M}$, 
and ${}^C\mathcal{M}$ has enough projectives if and only if $C$ is left semiperfect. 
\end{proof}

\begin{theor} \label{t:r} Let $\A$ be an abelian category and let $r$ be a preradical of $\A$.
\begin{enumerate}
\item Assume that $\A$ has an injective cogenerator $G$. Then the following are equivalent:
\begin{enumerate}[(i)]
\item $G$ is a (strongly) self-$r(G)$-split object of $\A$.
\item $G=r(G)\oplus D$ for some object $D$ of $\A$ and ${\rm End}_{\A}(D)$ is a von Neumann regular (strongly regular) ring.
\item Every finite direct product $M=G^n$ is (strongly) self-$r(M)$-split.
\item Every finitely cogenerated injective object $M$ of $\A$ is (strongly) self-$r(M)$-split.
\end{enumerate}
\item Assume that $\A$ has a projective generator $G$. Then the following are equivalent:
\begin{enumerate}[(i)]
\item $G$ is a dual (strongly) self-$r(G)$-split object.
\item $r(G)$ is a direct summand of $G$ and ${\rm End}_{\A}(r(G))$ is a von Neumann regular (strongly regular) ring.
\item Every finite direct sum $M=G^n$ is dual (strongly) self-$r(M)$-split.
\item Every finitely generated projective object $M$ of $\A$ is dual (strongly) self-$r(M)$-split.
\end{enumerate}
\end{enumerate}
\end{theor}

\begin{proof} (1) We first prove the theorem without strongness conditions. 

(i)$\Rightarrow$(ii) Assume that (i) holds. Then $G=r(G)\oplus D$ for some self-Rickart object $D$ of $\A$ 
by Theorem \ref{t:key}. We claim that $D$ is also dual self-Rickart.
To this end, let $f:D\to D$ be a morphism in $\A$. Then ${\rm Im}(f)\cong D/{\rm Ker}(f)$ is isomorphic to 
a direct summand of $D$. As $G$ is injective, so is $D$. Hence $D$ is direct injective, 
in the sense that every subobject of $D$ isomorphic to a direct summand of $D$ is a direct summand of $D$.
It follows that ${\rm Im}(f)$ is a direct summand of $D$. Hence $D$ is self-Rickart and dual self-Rickart. 
Then ${\rm End}_{\A}(D)$ is von Neumann regular by \cite[Corollary~2.3]{C-K}.

(ii)$\Rightarrow$(iii) Assume that (ii) holds. Let $M=G^n$ for some $n\in \mathbb{N}$. 
Then $M=r(G)^n\oplus D^n=r(M)\oplus D^n$. Note that ${\rm End}_{\A}(D^n)\cong {\rm End}_{\A}(D)^n$ is von Neumann regular, 
hence $D^n$ is self-regular, and so $D^n$ is self-Rickart by \cite[Corollary~2.3]{C-K}. 
Finally, $M$ is self-$r(M)$-split by Theorem \ref{t:key}.

(iii)$\Rightarrow$(iv) Assume that (iii) holds. Note that every finitely cogenerated injective object of $\A$ 
is isomorphic to a direct summand of a direct product $G^n$. Then use Corollary \ref{c:strel}.

(iv)$\Rightarrow$(i) This is clear. 

Finally, the equivalence of the strong versions of the above conditions is immediately deduced by the above proof 
and the following properties: an object $M$ of $\A$ is strongly self-$r(M)$-split if and only if 
$M$ is self-$r(M)$-split and ${\rm End}_{\A}(r(M))$ is abelian (Theorem \ref{t:endab}), while 
a ring is strongly regular if and only if it is von Neumann regular and abelian. 
\end{proof}

\begin{cor} Let $r$ be a preradical of ${}^C\mathcal{M}$. Then the following are equivalent:
\begin{enumerate}[(i)]
\item $C$ is a (strongly) self-$r(C)$-split left $C$-comodule.
\item $C=r(C)\oplus D$ for some subcomodule $D$ of $C$ and ${\rm End}_C(D)$ is a von Neumann regular (strongly regular) ring.
\item Every finitely cogenerated free left $C$-comodule $M$ is (strongly) self-$r(M)$-split.
\item Every finitely cogenerated injective left $C$-comodule $M$ is (strongly) self-$r(M)$-split.
\end{enumerate}
\end{cor}

\begin{proof} This follows by Theorem \ref{t:r}, noting that $C$ is an injective cogenerator of ${}^C\mathcal{M}$. 
\end{proof}

\begin{cor} Let $r$ be a preradical of ${\rm Mod}(R)$. Then the following are equivalent:
\begin{enumerate}[(i)]
\item $R$ is a dual (strongly) self-$r(R)$-split right $R$-module.
\item $r(R)$ is a direct summand of $R$ and ${\rm End}_R(r(R))$ is a von Neumann regular (strongly regular) ring.
\item Every finitely generated free right $R$-module $M$ is dual (strongly) self-$r(M)$-split.
\item Every finitely generated projective right $R$-module $M$ is dual (strongly) self-$r(M)$-split.
\end{enumerate}
\end{cor}

\begin{proof} This follows by Theorem \ref{t:r}, noting that $R$ is a projective generator of ${\rm Mod}(R)$. 
\end{proof}

\section{Further applications}

Throughout the paper we have given several examples and consequences of our results. 
In this section we present some further applications to module and comodule categories.

\subsection{Module categories} Throughout $R$ is a unitary ring. 

\begin{prop} 
\begin{enumerate} 
\item  Let $I$ be a right $T$-nilpotent ideal of $R$ and let $r$ be the cohereditary radical associated to $I$, given by $r(M)=MI$. 
Then a right $R$-module $M$ is (strongly) self-$r(M)$-split 
if and only if $M$ is $r$-torsionfree (strongly) self-Rickart. 
\item Let $I$ be a left $T$-nilpotent ideal of $R$ and let $r$ be the superhereditary preradical associated to $I$, given by $r(M)=\{x\in M\mid xI=0\}$. 
Then a right $R$-module $M$ is dual (strongly) self-$r(M)$-split if and only if $M$ is $r$-torsion (strongly) dual self-Rickart. 
\end{enumerate}
\end{prop}

\begin{proof} (1) There is a bijective correspondence between the ideals $I$ of a ring $R$ 
and the cohereditary radicals $r$ of ${\rm Mod}(R)$, where $r$ is given by $r(M)=MI$ \cite[Corollary~I.2.11]{BKN}.
If $M$ is (strongly) self-$r(M)$-split, then $r(M)$ is a superfluous direct summand of $M$ 
by \cite[Proposition~I.4.6]{BKN} and Theorem \ref{t:key}, hence $r(M)=0$, 
and so $M$ is $r$-torsionfree (strongly) self-Rickart. The converse is clear.  
 
(2) There is a bijective correspondence between the ideals $I$ of a ring $R$ 
and the superhereditary preradicals $r$ of ${\rm Mod}(R)$, where $r$ is given by 
$r(M)=\{x\in M\mid xI=0\}$ \cite[I.2.E6]{BKN}. 
If $M$ is dual (strongly) self-$r(M)$-split, then $r(M)$ is an essential direct summand of $M$ 
by \cite[Proposition~I.4.12]{BKN} and Theorem \ref{t:key}, hence $r(M)=M$, 
and so $M$ is $r$-torsion (strongly) dual self-Rickart. The converse is clear.   
\end{proof}

For a right $R$-module $M$, denote by ${\rm Soc}(M)$ the socle of $M$ (i.e., the sum of its simple submodules),
and by ${\rm Rad}(M)$ the Jacobson radical of $M$ (i.e., the intersection of its maximal submodules).
Then Soc is a hereditary preradical and Rad is a cohereditary preradical of ${\rm Mod}(R)$.

\begin{prop} 
\begin{enumerate} 
\item A finitely generated right $R$-module $M$ is (strongly) self-${\rm Rad}(M)$-split if and only if 
$M$ is semiprimitive (strongly) self-Rickart.
\item A finitely cogenerated right $R$-module $M$ is dual (strongly) self-${\rm Soc}(M)$-split if and only if 
$M$ is semisimple.
\end{enumerate}
\end{prop}

\begin{proof} (1) If $M$ is finitely generated (strongly) self-${\rm Rad}(M)$-split, 
then ${\rm Rad}(M)$ is a superfluous direct summand of $M$ by Theorem \ref{t:key}, hence ${\rm Rad}(M)=0$, 
and so $M$ is semiprimitive (strongly) self-Rickart. The converse is clear.  

(2) If $M$ is finitely cogenerated dual (strongly) self-${\rm Soc}(M)$-split, 
then ${\rm Soc}(M)$ is an essential direct summand of $M$ by Theorem \ref{t:key},  
hence $M={\rm Soc}(M)$ is semisimple. The converse is clear.  
\end{proof}

\begin{cor} 
\begin{enumerate} 
\item Every finitely generated right $R$-module $M$ is (strongly) self-${\rm Rad}(M)$-split if and only if 
every finitely generated right $R$-module $M$ is semiprimitive (strongly) self-Rickart.
\item Every finitely cogenerated right $R$-module $M$ is dual (strongly) self-${\rm Soc}(M)$-split if and only if 
$R$ is a right $V$-ring.
\end{enumerate}
\end{cor}

\begin{proof} (2) Note that every finitely cogenerated right $R$-module is semisimple if and only if 
$R$ is a right $V$-ring \cite[23.1]{Wis}. 
\end{proof}

Recall that an object $M$ of an abelian category $\A$ is called \emph{(strongly) extending} 
if every subobject of $M$ is essential in a (fully invariant) direct summand of $M$, 
and \emph{(strongly) lifting} if every subobject $L$ of $M$ contains a (fully invariant) direct summand $K$ of $M$ 
such that $L/K$ is superfluous in $M/K$ \cite{C-K,Atani14,WangY}. 

For a right $R$-module $M$, consider the full subcategory $\sigma[M]$ of ${\rm Mod}(R)$ consisting of 
$M$-subgenerated right $R$-modules. Note that if $M=R_R$, then $\sigma[M]={\rm Mod}(R)$. 
For every module $N\in \sigma[M]$, denote by:
\begin{itemize}
\item $Z_M(N)={\rm Tr}(\mathcal{U},N)=\sum \{{\rm Im}(f) \mid f\in \Hom(\mathcal{U},N)\}$, 
where $\mathcal{U}$ is the class of $M$-singular modules. 
Recall that a module $A$ is called \emph{$M$-singular} (or \emph{singular} in $\sigma[M]$) 
if $A\cong L/K$ for some $L\in \sigma[M]$ and essential submodule $K$ of $L$. 
Also, $A$ is called \emph{non-$M$-singular} (or \emph{non-singular} in $\sigma[M]$) if $Z_M(N)=0$.
\item $Z^2_M(N)$ the second $M$-singular submodule of $N$, which is determined by the equality
$Z^2_M(N)/Z_M(N)=Z_M(N/Z_M(N))$. 
\item ${\overline Z}_M(N)={\rm Re}(N,\mathcal{U})=\bigcap \{{\rm Ker}(f) \mid f\in \Hom(N,\mathcal{U})\}$, 
where $\mathcal{U}$ is the class of $M$-small modules. Recall that a module $A$ is called \emph{$M$-small} 
if $A$ is superfluous in some $L\in \sigma[M]$.
Also, $N$ is called \emph{$M$-cosingular} (or \emph{cosingular} in $\sigma[M]$) 
if ${\overline Z}_M(N)=0$, and \emph{non-$M$-cosingular} (or \emph{non-cosingular} in $\sigma[M]$) if ${\overline Z}_M(N)=N$.
\item ${\overline Z}^2_M(N)={\overline Z}_M({\overline Z}_M(N))$. 
\end{itemize}
Then $Z_M$ is a hereditary preradical, $Z^2_M$ is a hereditary radical, while
${\overline Z}_M$ and ${\overline Z}^2_M$ are radicals of $\sigma[M]$.

\begin{prop} 
\begin{enumerate}[(i)]
\item Assume that $R$ is non-singular (i.e., $Z(R_R)=0$) and every finitely generated non-singular right
$R$-module is projective. Then every finitely generated right $R$-module $M$ is self-$Z^2(M)$-split.
\item Assume that every finitely cogenerated singular right $R$-module
is injective. Then every finitely cogenerated right $R$-module $M$ is dual self-$Z(M)$-split.
\end{enumerate}
\end{prop}

\begin{proof} (i) Let $M$ be a finitely generated right $R$-module.
Since the finitely generated non-singular (see \cite[Lemma 3.2]{V})
right $R$-module $M/Z^2(M)$ is projective, $M\cong Z^2(M)\oplus
M/Z^2(M)$. Since every ring with the property from hypothesis is
right semihereditary \cite[Theorem~5.18]{Goodearl}, $M/Z^2(M)$ is
self-Rickart \cite[Theorem~4.7]{C-K}. Hence $M$ is
self-$Z^2(M)$-split by Theorem \ref{t:key}. 

(ii) Let $M$ be a finitely cogenerated right $R$-module. Since the
finitely cogenerated singular right $R$-module $Z(M)$ is
injective, it is a direct summand of $M$. We claim that $R$ is
right cosemihereditary. To this end, let $E$ be an injective right
$R$-module and $K$ a submodule of $E$ such that $E/K$ is finitely
cogenerated. Consider the injective hull $E(K)$ of $K$ and the
induced epimorphism $\gamma:E/K\to E/E(K)$. Then ${\rm
Ker}(\gamma)\cong E(K)/K$ is finitely cogenerated singular, and so
it is injective. Also, we have $E\cong E(K)\oplus E/E(K)$, and so
$E/E(K)$ is injective. It follows that $E/K$ is injective, which
shows that $R$ is right cosemihereditary. Then $Z(M)$ is dual
self-Rickart \cite[Theorem~4.7]{C-K}. Hence $M$ is dual
self-$Z(M)$-split by Theorem~\ref{t:key}.
\end{proof}

\begin{prop} Let $M$ be a right $R$-module. Then:
\begin{enumerate}
\item Every (strongly) extending module $N\in \sigma[M]$ is (strongly) self-$Z^2_M(N)$-split.
\item Every (strongly) lifting module $N\in \sigma[M]$ is dual (strongly) self-${\overline Z}^2_M(N)$-split.
\end{enumerate}
\end{prop}

\begin{proof} (1)  If $N$ is (strongly) extending, then $N=Z^2_M(N)\oplus N'$ for some non-$M$-singular (strongly) extending 
submodule $N'$ of $N$ \cite[Theorem~3.4]{V} (\cite[Theorem~2.24]{Atani14}). 
Note that a module $X\in \sigma[M]$ is non-$M$-singular if and only if 
for every $K\in \sigma[M]$ and for every $0\neq f\in \Hom(K,N')$, ${\rm Ker}(f)$ is not essential in $K$.  
Clearly, every non-$M$-singular module $X\in \sigma[M]$ is $X$-$\mathcal{K}$-nonsingular, 
in the sense that for every $0\neq f\in {\rm End}(X)$, ${\rm Ker}(f)$ is not essential in $X$ \cite[Definition~9.4]{C-K}. 
It follows that $N'$ is (strongly) self-Baer by \cite[Theorem~9.5]{C-K}, and so $N'$ is (strongly) self-Rickart. 
Hence $N$ is (strongly) self-$Z^2_M(N)$-split by Theorem \ref{t:key}. 

(2) If $N$ is (strongly) lifting, then $N={\overline Z}^2_M(N)\oplus N'$ for some submodule $N'$ of $N$ such that 
${\overline Z}^2_M(N)$ is non-$M$-cosingular (strongly) lifting \cite[Theorem~4.1]{TV} (\cite[Corollary~3.6]{WangY}). 
Note that a module $Y\in \sigma[M]$ is non-$M$-cosingular if and only if 
every non-zero factor module of $Y$ is not $M$-small, that is, it does not exist any module $L\in \sigma[M]$ 
such that $Y$ is superfluous in $L$. Clearly, every non-$M$-cosingular module $Y\in \sigma[M]$ is 
$Y$-$\mathcal{T}$-nonsingular, in the sense that 
for every $0\neq f\in {\rm End}(Y)$, ${\rm Im}(f)$ is not superfluous in $Y$ \cite[Definition~9.4]{C-K}.
It follows that ${\overline Z}^2_M(N)$ is dual (strongly) self-Baer by \cite[Theorem~9.5]{C-K}, 
and so ${\overline Z}^2_M(N)$ is dual (strongly) self-Rickart. 
Hence $N$ is dual (strongly) self-${\overline Z}^2_M(N)$-split by Theorem \ref{t:key}. 
\end{proof}

\subsection{Comodule categories} Throughout $C$ is a coalgebra over a field $k$. 
Let $C^*=\Hom_k(C,k)$. Then $\sigma[C_{C^*}]={\rm Rat}({\rm Mod}(C^*))\cong {}^C\mathcal{M}$, where 
${\rm Rat}({\rm Mod}(C^*))$ is the full subcategory of ${\rm Mod}(C^*)$ consisting of rational right $C^*$-modules 
and ${}^C\mathcal{M}$ is the category of left $C$-comodules \cite[Chapter~2]{DNR}. 
Hence one has the (pre)radicals from categories of the form $\sigma[M]$ for some module $M$ available in comodule categories. 
Also, for every right $C^*$-module $M$, denote by ${\rm Rat}(M)$ the largest rational submodule of $M$.   
Then Rat is a hereditary preradical of ${\rm Mod}(C^*)$.

\begin{prop} Let $M$ be a left $C$-comodule.
Then the following are equivalent:
\begin{enumerate}[(i)] 
\item $M$ is semisimple. 
\item $M$ is (strongly) self-${\rm Soc}(M)$-split.
\item $M$ is dual (strongly) self-${\rm Soc}(M)$-split.
\end{enumerate}
\end{prop}

\begin{proof} If $M$ is non-zero (strongly) self-${\rm Soc}(M)$-split or dual (strongly) self-${\rm Soc}(M)$-split, then 
${\rm Soc}(M)$ is a direct summand of $M$ by Theorem \ref{t:key} and 
an essential subcomodule of $M$ by \cite[Corollary~2.4.12]{DNR}, 
which implies that $M={\rm Soc}(M)$ is semisimple by \cite[Proposition~2.4.11]{DNR}. 
The rest of the proof is clear. 
\end{proof}

\begin{cor} The following are equivalent:
\begin{enumerate}[(i)]
\item $C$ is cosemisimple.
\item Every left $C$-comodule $M$ is (strongly) self-${\rm Soc}(M)$-split.
\item Every left $C$-comodule $M$ is dual (strongly) self-${\rm Soc}(M)$-split.
\end{enumerate}
\end{cor}

\begin{prop} Let $C$ be such that $C^*$ is semisimple. Then the following are equivalent:
\begin{enumerate}[(i)]
\item $C$ is cosemisimple.
\item Every right $C^*$-module $M$ is (strongly) self-${\rm Rat}(M)$-split.
\item Every right $C^*$-module $M$ is dual (strongly) self-${\rm Rat}(M)$-split.
\end{enumerate}
\end{prop}

\begin{proof} If every right $C^*$-module $M$ is (strongly) self-${\rm Rat}(M)$-split or 
dual (strongly) self-${\rm Rat}(M)$-split, then 
${\rm Rat}(M)$ is a direct summand of $M$ for every every right $C^*$-module $M$ by Theorem \ref{t:key}. 
Then $C$ is finite dimensional by \cite[Theorem~3.4]{NT}. Since $C^*$ is semisimple, 
it follows that $C$ is cosemisimple \cite[Exercise~3.1.7]{DNR}. The rest of the proof is clear.  
\end{proof}

\begin{prop} Let $C$ be hereditary and let $P$ be a projective left $C$-comodule. 
Then $P^*=\Hom_k(P,k)$ is a dual self-${\rm Rat}({}_{C^*}P^*)$-split right $C$-comodule. 
\end{prop}

\begin{proof} By \cite[Corollary~2.4.18]{DNR}, ${\rm Rat}({}_{C^*}P^*)$ is an injective right $C$-comodule. 
Then ${\rm Rat}({}_{C^*}P^*)$ is a direct summand of $P^*$ and dual self-Rickart by \cite[Corollary~4.9]{C-K}. 
Hence $P^*$ is a dual self-${\rm Rat}({}_{C^*}P^*)$-split right $C$-comodule by Theorem \ref{t:key}.
\end{proof}

\begin{prop} Let $C$ be left semiperfect hereditary and let $F$ be a flat left $C$-comodule. 
Then $F^*=\Hom_k(F,k)$ is a dual self-${\rm Rat}({}_{C^*}F^*)$-split right $C$-comodule. 
\end{prop}

\begin{proof} By \cite[Theorem~4.6]{CS07} and \cite[Corollary~3.4]{CPT}, 
$C$ is left semiperfect if and only if ${\rm Rat}({}_{C^*}F^*)$ is an injective right $C$-comodule 
for every every flat left $C$-comodule $F$. Then ${\rm Rat}({}_{C^*}F^*)$ is a direct summand of $F^*$ 
and dual self-Rickart by \cite[Corollary~4.9]{C-K}. 
Hence $F^*$ is a dual self-${\rm Rat}({}_{C^*}F^*)$-split right $C$-comodule by Theorem \ref{t:key}.
\end{proof}

\end{document}